\documentclass[11pt]{amsart}

\usepackage{graphicx}
\usepackage{epsfig}
\usepackage{amsmath}
\usepackage{autobreak}
\usepackage{amssymb}
\usepackage{bigints}
\usepackage[]{breqn}
\usepackage{tikz}
\usepackage{graphicx}
\usepackage{epsfig}
\usepackage[noadjust]{cite}
\usepackage{xcolor}
\usepackage{tikz}
\usetikzlibrary{matrix}
\usepackage[colorlinks,citecolor=red,urlcolor=blue,bookmarks=false,hypertexnames=true]{hyperref}

\theoremstyle{plain}
\newtheorem{theorem}      {Theorem}      [section]
\newtheorem*{theorem*}    {Theorem \eqref{thm:appl}}
\newtheorem{proposition}  [theorem]  {Proposition}

\theoremstyle{definition}
\newtheorem{example}      [theorem]  {Example}
\newtheorem{remark}       [theorem]  {Remark}

\def \i{\mbox{${\textnormal{i}}$}}
\def \d{\mbox{${\textnormal{d}}$}}
\def \dz{\mbox{${\stackrel{o}{\textnormal{d}}}$}}
\def \Dz{\mbox{${\stackrel{o}{\Delta}}$}}
\def \nz{\mbox{${\stackrel{o}{\nabla}}$}}
\def \tz{\mbox{${\stackrel{o}{\tau}}$}}

\def \t{\mbox{${\mathbb T}$}}
\def \r{\mbox{${\mathbb R}$}}
\def \s{\mbox{${\mathbb S}$}}

\DeclareMathOperator{\trace}{trace} 

\numberwithin{equation}{section}

\begin{document}

\title[Biharmonic homogeneous polynomial maps between spheres]
{Biharmonic homogeneous polynomial maps between spheres}

\author{Rare\c s Ambrosie, Cezar Oniciuc, Ye-Lin Ou}

\address{Faculty of Mathematics\\ Al. I. Cuza University of Iasi\\
Blvd. Carol I, 11 \\ 700506 Iasi, Romania} \email{rares_ambrosie@yahoo.com}

\address{Faculty of Mathematics\\ Al. I. Cuza University of Iasi\\
Blvd. Carol I, 11 \\ 700506 Iasi, Romania} \email{oniciucc@uaic.ro}

\address{Department of Mathematics\\ Texas A $\&$ M University-Commerce\\
Commerce TX 75429 \\ USA} \email{yelin.ou@tamuc.edu}

\thanks{Y.-L. Ou was supported by a grant from the Simons Foundation ($\#427231$, Ye-Lin Ou), which included a travel fund for a short-visit of the second named author to Texas A $\&$ M University-Commerce in 2018 during which a part of the work was done.}

\thanks{Thanks are due to Paul Baird and Andrea Ratto for useful comments and remarks.}



\begin{abstract}
In this paper we first prove a characterization formula for biharmonic maps in Euclidean spheres and, as an application, we construct a family of biharmonic maps from a flat $2$-dimensional torus $\mathbb{T}$ into the $3$-dimensional unit Euclidean sphere $\mathbb{S}^3$. Then, for the special case of maps between spheres whose components are given by homogeneous polynomials of the same degree, we find a more specific form for their bitension field. Further, we apply this formula to the case when the degree is $2$, and we obtain the classification of all proper biharmonic quadratic forms from $\s^1$ to $\s^n$, $n \geq 2$, from $\s^m$ to $\s^2$, $m \geq 2$, and from $\s^m$ to $\s^3$, $m \geq 2$.
\end{abstract}

\maketitle

\section{Introduction}

Biharmonic maps represent a natural generalization of the well-known harmonic maps. As suggested by J. Eells and J.H. Sampson in  \cite{ES64,ES65}, or J. Eells and L. Lemaire in \cite{EL83}, biharmonic maps $\phi:M^m\to N^n$ between two Riemannian manifolds are critical points of the bienergy functional
$$
E_2:C^\infty(M,N)\to \mathbb{R}, \qquad E_2(\phi) = \frac{1}{2}\int_{M}\left|\tau(\phi)\right|^2 \ v_g,
$$
where $M$ is compact and $\tau(\phi) = \trace\nabla d\phi$ is the tension field associated to $\phi$. In 1986, G.Y. Jiang proved in  \cite{J86,J86-2} that the biharmonic maps are characterized by the vanishing of their bitension field, where the associated bitension field is given by
$$
\tau_2(\phi) = -\Delta\tau(\phi)-\trace R^N\left(d\phi(\cdot),\tau(\phi)\right)d\phi(\cdot).
$$
The equation $\tau_2(\phi) = 0$ is called the biharmonic equation and it is a forth order semilinear elliptic equation.

Any harmonic map is obviously biharmonic, so we are interested in the study of biharmonic maps which are not harmonic, called proper biharmonic. G.Y. Jiang proved in \cite{J86,J86-2}, using a simple Bochner-Weitzenb$\ddot{o}$ck formula, that if $M$ is compact and $N$ has non-positive sectional curvature, then a biharmonic map $\phi$ from $M$ to $N$ has to be harmonic. Thus, the effort has been focused on the study of biharmonic maps in spaces of positive sectional curvature, in particular in Euclidean spheres, spaces of constant positive sectional curvature. In the case of isometric immersions, many examples and classification results of proper biharmonic submanifolds in Euclidean spheres have been obtained (see, for example, \cite{FO22}, \cite{O12} and \cite{OC20}).

Apart from biharmonic isometric immersions, and also conformal immersions, there has been studied the biharmonicity of maps $\phi:M\rightarrow\s^n$ that are originated from harmonic submersions $\psi$ defined on $M$ and with values in a small hypersphere $\s^{n-1}(r)$ of radius $r$ of $\s^n$, and examples of such proper biharmonic maps were obtained considering $\psi$ as being the classical harmonic Hopf fibration $H:\s^{2(n-1)-1}\rightarrow\s^{n-1}(1/\sqrt{2})$, $n = 3$, $5$, $9$, (see \cite{O03}). Inspired by these examples whose components, as functions in $\r^{n+1}$, are homogeneous polynomials of degree $2$, in this paper we study the biharmonicity of the maps $\phi:\s^m\rightarrow\s^n$ whose components are non-harmonic homogeneous polynomials of degree $k$. We note that the homogeneous polynomial maps between Euclidean spheres are very important in the harmonic map theory because, when they are harmonic as functions between Euclidean spaces, they provide all harmonic maps with constant energy density between the corresponding Euclidean spheres of codimension $1$ (see, for example, \cite{BW03} and \cite{ER93}). Especially when $k = 2$, many interesting results have been obtained in \cite{HMX03}, \cite{L85}, \cite{PR90}, \cite{W04}, \cite{WOY14}, \cite{WZ13} etc.

An interesting problem is to study biharmonic maps in certain classes of maps where it is known that there are no harmonic ones. For example, it well known that there exist no full harmonic maps from $\s^4$ to $\s^5$, such  that their components are homogeneous polynomials of degree $2$. By way of contrasts, using an example of a full harmonic map from $\s^4$ to $\s^4$ and the method in \cite{LO07}, we can construct a full biharmonic map $\phi:\s^4\rightarrow\s^5$ of the above type (see Theorem \ref{t45}).

In the first part of our paper, we give a characterisation formula of the biharmonicity in the most general case of maps in Euclidean spheres (see Theorem \ref{t24}). Then, as an application, we start with the standard minimal Clifford torus in $\s^3$ and, performing a simple transformation of the domain, we get a proper biharmonic map in $\s^3$ (see Proposition \ref{p26}). In the next section, we consider the case of maps between spheres such that each of their components is a homogeneous polynomial of degree $k$, and we compute their bitension field (see Theorem \ref{t3}). Further, we particularize Theorem \ref{t3} to the case $k = 2$, i.e. to the quadratic forms, and we obtain a simpler expression of the bitension field (see Theorem \ref{t4}). In the last part of the paper, using the known result of R. Wood for the characterisation of quadratic forms between spheres (see \cite{W68}), we obtain the classification of all proper biharmonic quadratic forms from $\s^1$ to $\s^n$, $n \geq 2$ (see Theorem \ref{t47}), from $\s^m$ to $\s^2$, $m \geq 2$ (see Theorem \ref{t48}), and from $\s^m$ to $\s^3$, $m \geq 2$ (see Theorem \ref{t4.6}).  In order to obtain Theorems \ref{t48} and \ref{t4.6}, as an alternative method, one can use the known results obtained by S. Chang in \cite{C98}, P.Y.H. Yiu in \cite{Y86} and R. Wood in \cite{W68}, to first determine all quadratic forms from $\s^m$ to $\s^2$ and from $\s^m$ to $\s^3$, and then to determine the biharmonic ones. However, our approach is more direct and simpler. We end the paper with an Open Problem concerning the structure of proper biharmonic quadratic forms between spheres.

\textbf{Conventions.} We use the following sign conventions for the rough Laplacian, that acts on the set $C\left(\phi^{-1}TN\right)$ of all sections of the pull-back bundle $\phi^{-1}TN$, and for the curvature tensor field
$$
\Delta \sigma = -\textnormal{trace}_g\nabla^2\sigma, \quad R(X,Y)Z = \nabla_X\nabla_Y Z - \nabla_Y\nabla_YZ -\nabla_{[X,Y]}Z.
$$
Also, by $\s^m(r)$ we indicate the $m$-dimensional Euclidean sphere of radius $r$. When $r = 1$, we write $\s^m$ instead of $\s^m(1)$.

\section{Biharmonic equations for maps between spheres}

First, we recall the following result

\begin{theorem}[\cite{LO07}]\label{t1}
Let M be a compact manifold and consider $\psi: M \rightarrow \s^n(r/\sqrt{2})$ a nonconstant map, where $\s^n(r/\sqrt{2})$ is a small hypersphere of radius $r/\sqrt{2}$ of $\s^{n+1}(r)$. The map $\phi = \i\circ\psi : M\rightarrow \s^{n+1}(r)$, where $\i$ is the canonical inclusion, is proper biharmonic if and only if $\psi$ is harmonic and the energy density $e(\psi)$ is constant.
\end{theorem}

\begin{remark}
As we can see from the proof, the compactness is needed only for the direct implication.
\end{remark}

\begin{remark}
When the map $\psi$ is an isometric immersion, i.e. $M$ is a submanifold, or a Riemannian submersion, the result was obtained in \cite{CMO02} and \cite{O03}, respectively.
\end{remark}

As an example of the above result, we can consider the map $\psi: \s^3(\sqrt{2}) \rightarrow \s^2(1/\sqrt{2})$ as being the classical Hopf map and then by composing with the inclusion of $\s^2(1/\sqrt{2})$ into $\s^3$, we obtain that
\begin{equation*}
\phi: \s^3(\sqrt{2})=\left\{(z^1,z^2)\in\mathbb{C}^2 \, : \, \left|z^1\right|^2 + \left|z^2\right|^2 = 2 \right\} \rightarrow \s^3
\end{equation*}
given by
\begin{align}\label{ec1}
\phi(z^1,z^2) =& \frac{1}{2\sqrt{2}} \left(2z^1\overline{z^2}, \left|z^1\right|^2 - \left|z^2\right|^2, 2\right) \\
              =& \frac{1}{2\sqrt{2}} \left(2z^1\overline{z^2}, \left|z^1\right|^2 - \left|z^2\right|^2, \left|z^1\right|^2 + \left|z^2\right|^2\right) \nonumber
\end{align}
is a proper biharmonic map (see \cite{O03}). As homothetic changes of the domain or codomain metrics preserves the harmoncity and biharmonicity, we can assume that $\phi$ maps $\s^3$ into $\s^3$. We also note that the components of $\phi$ are (restrictions of) homogeneous polynomials of degree $2$.

We also recall the following non-existence result for the harmonic maps
\begin{theorem}[\cite{WZ13}]
There are no full harmonic maps from $\s^4$ to $\s^n$, for $n=5$, $6$, whose components are homogeneous polynomials of degree $2$.
\end{theorem}
Here a map $\phi:\s^m\rightarrow\s^n$ is said to be full if its image is not contained in any great sphere.

On the other hand, as presented in \cite{C39} (see also \cite{T10}), it is well known that the gradient of the isoparametric function
\begin{align*}
f(x^1,x^2,x^3,x^4,x^5) =& \frac{1}{3}\left(x^1\right)^3 + \frac{x^1}{2}\left(\left(x^3\right)^2 + \left(x^4\right)^2 - 2\left(x^5\right)^2 - 2\left(x^2\right)^2\right)\\
                        & + \frac{\sqrt{3}}{2}\left(x^2\left(x^3\right)^2 - x^2\left(x^4\right)^2\right) + \sqrt{3} x^3x^4x^5.
\end{align*}
gives, up to a homothethic change of the codomain, a full harmonic map $\psi:\s^4\rightarrow\s^4(1/\sqrt{2})$, with constant energy density, defined by
\begin{align*}
\psi(x^1,x^2,x^3,x^4,x^5) =& \frac{1}{\sqrt{2}}\left( \left(x^1\right)^2 + \frac{1}{2}\left(x^3\right)^2 + \frac{1}{2}\left(x^4\right)^2 - \left(x^5\right)^2 - \left(x^2\right)^2,\right. \\
                           & \left. -2x^1x^2 + \frac{\sqrt{3}}{2}\left(x^3\right)^2 - \frac{\sqrt{3}}{2}\left(x^4\right)^2, x^1x^3 + \sqrt{3}x^2x^3 + \sqrt{3}x^4x^5, \right. \\
                           & \left. x^1x^4 - \sqrt{3}x^2x^4 +  \sqrt{3}x^3x^5, -2x^1x^5 +  \sqrt{3}x^3x^4\right).
\end{align*}
Then, for the biharmonic case, using Theorem  \ref{t1} from above we conclude that
\begin{theorem}\label{t45}
There exist full proper biharmonic maps $\s^4$ to $\s^5$, whose components are homogeneous polynomials of degree $2$.
\end{theorem}

Now, let $\phi:(M^m,g)\to\s^n$ be a map and let $\i:\s^n\to \r ^{n+1}$ be the standard isometric embedding of the unit Euclidean sphere into Euclidean space. We consider the composition map
$$
\varphi = \i \circ \phi:(M^m,g)\to \r ^{n+1}.
$$
As usual, we identify locally $\d\phi(X)$ with $\d\varphi(X)$, for any vector fields $X$ tangent to $M$.

We denote $\theta = \left\langle\d \phi, \tau(\phi)\right\rangle = \left\langle\d \varphi, \tau(\varphi)\right\rangle$, i.e. $\theta$ is the $1$-form on $M$ given by
$$
\theta(X) = \left\langle\d \phi(X), \tau(\phi)\right\rangle = \left\langle\d \varphi(X), \tau(\varphi)\right\rangle.
$$

By $\theta^\sharp$ we denote the vector field associated to the $1$-form $\theta$ through the musical isomorphism. With these notations, we can state:

\begin{theorem}\label{t24}
Let $\phi:(M^m,g)\to \s^n$ be a map and let $\i:\s^n\to \r ^{n+1}$ be the standard isometric embedding. Then, $\phi$ is a biharmonic map if and only if the vector function $\varphi = \i \circ \phi:(M^m,g)\to \r ^{n+1}$ solves the following PDE
\begin{equation}\label{eq1}
\tau_2(\varphi) + 2|\d  \varphi|^2 \tau(\varphi) + \left( -\Delta |\d  \varphi|^2 + 2\textnormal{div}\theta^\sharp - |\tau(\varphi)|^2 + 2|\d \varphi|^4 \right ) \varphi + 2\d \varphi\left(\textnormal{grad}|\d \varphi|^2 \right) = 0.
\end{equation}
\end{theorem}

\begin{proof}
It is well-known that the standard isometric embedding $\i:\s^n\to \r ^{n+1}$ is a totally umbilical hypersurface with the unit normal vector field $r$. The vector field $r$ associates to any point the corresponding position vector.

Since $\nabla_U^{\r ^{n+1}}r = U$, for any $U\in C(T\r ^{n+1})$, we get that the second fundamental form and the shape operator of $\s ^n$ into $\r ^{m+1}$ are given by
$$
B(X,Y) = -\left\langle X,Y \right\rangle  r\quad\textnormal{and}\quad A_r(X) = A(X) = -X,
$$
for any $X,Y\in C(T\s ^n)$.

Let $\phi:M\to \s ^n$ be a smooth map and $\varphi = \textnormal{i}\circ \phi:M\to \r ^{n+1}$ the composition map.
We note that if $\sigma \in C(\phi^{-1}T\s ^n),$ then di$(\sigma) \in C(\varphi^{-1}T\r ^{n+1})$ and
\begin{equation}\label{eq2}
\nabla_X^\varphi \textnormal{di}(\sigma) = \textnormal{di}\left(\nabla_X^\phi\sigma\right) - \left\langle \d\phi(X),\sigma\right\rangle r\circ\varphi.
\end{equation}

In order to simplify the notation, we will not mention anymore the action of di and we will write $\varphi$ instead of $ r\circ\varphi.$

From the Equation \eqref{eq2} we quickly deduce that
\begin{equation}\label{eq3}
\tau(\varphi) = \tau(\phi) - |\d \phi|^2\varphi = \tau(\phi) - |\d \varphi|^2\varphi.
\end{equation}

In order to compute $\tau_2(\phi)$ in terms of $\varphi$, we consider an arbitrary point $p\in M$ and a geodesic frame field $\left\{X_i\right\}^m_{i=1}$ around $p$. From \eqref{eq2}, around $p$ we have
\begin{equation}\label{eq4}
\nabla^\phi_{X_i} \tau(\phi) = \nabla^\varphi_{X_i}\tau(\phi) + \left\langle \d \phi(X_i), \tau(\phi)\right\rangle \varphi.
\end{equation}

Using the Equation \eqref{eq3}, the Equation \eqref{eq4} becomes
\begin{align}\label{eq5}
\nabla_{X_i}^\phi\tau(\phi) &= \nabla_{X_i}^\varphi\left\{ \tau(\varphi) + |\d \varphi|^2\varphi \right\} + \theta(X_i)\varphi\\
                            &= \nabla_{X_i}^\varphi\tau(\varphi) + \left\{X_i\left(|\d \varphi|^2\right)\right\}\varphi+|\d \varphi|^2 \d \varphi(X_i) + \left\langle X_i, \theta^\sharp\right\rangle\varphi. \nonumber
\end{align}

From Equation \eqref{eq2} with $\sigma = \nabla^\phi_{X_i}\tau(\phi)$, around $p$, we have
\begin{equation}\label{eq6}
\nabla^\phi_{X_i}\nabla^\phi_{X_i}\tau(\phi) = \nabla^\varphi_{X_i} \nabla^\phi_{X_i}\tau(\phi) + \left\langle\d \phi(X_i), \nabla^\phi_{X_i}\tau(\phi)\right\rangle\varphi.
\end{equation}

Replacing the Equation \eqref{eq5} in the Equation \eqref{eq6} we get
\begin{align}\label{eq7}
\nabla^\phi_{X_i}\nabla^\phi_{X_i} \tau(\phi) =& \nabla^\varphi_{X_i} \left\{ \nabla^\varphi_{X_i}\tau(\varphi) + \left( X_i\left(|\d \varphi|^2\right) \right)\varphi + |\d \varphi|^2\d \varphi(X_i)+\left\langle X_i,\theta^\sharp\right\rangle\varphi\right\} \nonumber \\
                                               & + \left\langle\d \phi(X_i), \nabla^\varphi_{X_i}\tau(\varphi) + \left( X_i\left(|\d \varphi|^2\right) \right)\varphi + |\d  \varphi|^2\d \varphi(X_i) + \left\langle X_i,\theta^\sharp\right\rangle\varphi \right\rangle\varphi \\
                                              =& \nabla^\varphi_{X_i}\nabla^\varphi_{X_i}\tau(\varphi) + \left( X_i X_i \left(|\d \varphi|^2\right) \right)\varphi + 2 \left(X_i \left(|\d \varphi|^2\right) \right)\d \varphi(X_i) + |\d \varphi|^2\nabla^\varphi_{X_i}\d \varphi(X_i) \nonumber\\
                                               & + \left( \left\langle\nabla_{X_i}X_i,\theta^\sharp\right\rangle + \left\langle X_i, \nabla_{X_i}\theta^\sharp\right\rangle \right)\varphi + \left\langle X_i,\theta^\sharp\right\rangle\d \varphi(X_i)\nonumber \\
                                               & + \left( X_i\left\langle\d \phi(X_i),\tau(\phi)\right\rangle - \left\langle\nabla^\varphi_{X_i}\d \varphi(X_i),\tau(\varphi)\right\rangle \right)\varphi + |\d \varphi|^2|\d \varphi(X_i)|^2\varphi.\nonumber
\end{align}
Taking the sum in Equation \eqref{eq7}, at $p$ we get
\begin{align*}
-\Delta^\phi\tau(\phi) =& -\Delta^\varphi\tau(\varphi) - \Delta\left(|\d \varphi|^2\right)\varphi + 2\d \varphi\left(\textnormal{grad}\left(|\d \varphi|^2\right)\right) + |\d \varphi|^2\tau(\varphi) \\
                        & + \left(\textnormal{div}\theta^\sharp\right)\varphi + \d \varphi(\theta^\sharp) + \left(\textnormal{div}\theta^\sharp\right)\varphi - |\tau(\varphi)|^2\varphi + |\d \varphi|^4\varphi \\
                       =& \tau_2(\varphi) +|\d \varphi|^2\tau(\varphi)+ \\
                        & + \left\{-\Delta\left(|\d \varphi|^2\right) + 2\left(\textnormal{div}\theta^\sharp\right) - |\tau(\varphi)|^2 + |\d \varphi|^4\right\}\varphi \\
                        & + 2\d \varphi\left(\textnormal{grad}\left(|\d \varphi|^2\right)\right) + \d \varphi(\theta^\sharp).
\end{align*}
Finally, replacing $\tau(\phi)$ and $\Delta^\phi\tau(\phi)$ in the expression of $\tau_2(\phi)$, we get
\begin{align*}
\tau_2(\phi) =& -\Delta^\phi\tau(\phi) - \textnormal{trace}R^{\s^n}\left(\d \phi, \tau(\phi)\right)\d \phi \\
             =& -\Delta^\phi\tau(\phi) - \sum_i\left\{ \left\langle\d \phi(X_i),\tau(\phi)\right\rangle\d \phi(X_i) - \left\langle\d \phi(X_i),\d \phi(X_i)\right\rangle\tau(\phi)\right\}\\
             =& -\Delta^\phi\tau(\phi)-\d \phi(\theta^\sharp) + |\d \varphi|^2\tau(\varphi)+|\d \varphi|^4\varphi.
\end{align*}
We conclude that
\begin{align}\label{ec3}
\tau_2(\phi) =& \tau_2(\varphi) +2|\d \varphi|^2\tau(\varphi)+ \nonumber\\
              & + \left\{-\Delta\left(|\d \varphi|^2\right) + 2\left(\textnormal{div }\theta^\sharp\right) - |\tau(\varphi)|^2 + 2|\d \varphi|^4\right\}\varphi\\
              & + 2\d \varphi\left(\textnormal{grad}\left(|\d \varphi|^2\right)\right),\nonumber
\end{align}
and the conclusion follows.
\end{proof}

\begin{remark}
We note that Equation \eqref{eq1} first appeared in \cite{OY18} and then, a version of this formula was given in \cite{BO19}. A similar form of Equation \eqref{eq1} was obtained in \cite{K08} and \cite{W04} for maps from Euclidean domains in spheres, and in \cite{CMO02} for isometric immersions in Euclidean spheres.
\end{remark}

As an application to Equation \eqref{eq1} we will give the follwing example. Let $\Psi: \s^1\times\s^1\to\s^3$ be a smooth map defined by $$\Psi(s,t)=\frac{1}{\sqrt2}(\cos s, \sin s, \cos t, \sin t).$$ This map is well-defined, harmonic, and
$$
\Psi^*g_{\s^3}=\frac{1}{2}g_\t,
$$
where $\t$ denotes $\s^1\times\s^1$ and $g_\t$ the standard metric, $g_\t = \d s^2 +\d t^2$. In fact, $\Psi$ represents, up to a homothetic change of domain metric, the standard minimal Clifford torus in $\s^3$.

In the following, we will perform a simple transformation of the domain in order to render the harmonic map $\Psi$ into a proper biharmonic map. More precisely, we will compose the map $\Psi$ with a totally geodesic map $G:\t\rightarrow\t$ such that $\Psi\circ G$ will be biharmonic but not harmonic.

Consider the map $G:\t\to\t$ given in standard coordinates by
$$
G(u,v)=(k_1u+k_2v,l_1u+l_2v),
$$
where $k_1,k_2,l_1,l_2 \in \mathbb{Z}$. The map $G$ is well-defined, non-necessarily injective, and totally geodesic. We have
$$
G_u = \d G(\partial_u) =  k_1\partial_s+l_1\partial_t \quad\textnormal{and}\quad G_v = \d G(\partial_v) =  k_2\partial_s+l_2\partial_t.
$$
Thus
\begin{align*}
  &\begin{cases}
    |G_u|^2 & = k_1^2+l_1^2\\
    |G_v|^2 & = k_2^2+l_2^2\\
    \left\langle G_u,G_v\right\rangle & = k_1k_2+l_1l_2.
   \end{cases}
\end{align*}

We note that the metric  $G^*g_\t$ is homothetic to $g_\t$ if and only if $k_1^2+l_1^2 = k_2^2+l_2^2 > 0$ and $k_1k_2+l_1l_2 = 0$, that is the vectors $(k_1,l_1)$ and $(k_2,l_2)$ have the same length and are orthogonal, i.e. $(k_2,l_2) = \pm(-l_1,k_1)$. In particular, if $G^*g_\t$ is homothetic to $g_\t$, then $l_1^2 + l_2^2 = k_1^2 + k_2^2$.

Denote $\phi:=\Psi\circ G$ and $\varphi:=\i\circ\phi:\t\to\r ^4,$
$$
\varphi(u,v) = \frac{1}{\sqrt2}\left(\cos(k_1u + k_2v),\sin(k_1u + k_2v),\cos(l_1u + l_2v),\sin(l_1u + l_2v)\right).
$$

With the above notation, we have that
\begin{proposition}\label{p26}
The map $\phi:\t\to\s^3$ is proper biharmonic if and only if
$$
l_1^2+l_2^2 \neq k_1^2+k_2^2.
$$
\end{proposition}
\begin{proof}
By direct computations we obtain
\begin{align*}
    &\begin{cases}
      |\varphi_u|^2 & = \frac{1}{2}(k_1^2+l_1^2), \\
      |\varphi_v|^2 & = \frac{1}{2}(k_2^2+l_2^2), \\
      \left\langle\varphi_u,\varphi_v\right\rangle & = \frac{1}{2}(k_1k_2+l_1l_2),
    \end{cases} \\
    &|\d \varphi|^2  =  |\varphi_u|^2+|\varphi_u|^2 = \frac{1}{2}(k_1^2+l_1^2+k_2^2+l_2^2).
\end{align*}
Now,
\begin{align*}
\Delta\varphi  =& \frac{1}{\sqrt2}\left((k_1^2+k_2^2)\cos(k_1u+k_2v),(k_1^2+k_2^2)\sin(k_1u+k_2v),\right. \\
                &  \left.(l_1^2+l_2^2)\cos(l_1u+l_2v),(l_1^2+l_2^2)\sin(l_1u+l_2v)\right).
\end{align*}
As a result,
\begin{equation*}
|\Delta\varphi|^2 =\frac{1}{2}\left((k_1^2+k_2^2)^2+(l_1^2+l_2^2)^2\right)
\end{equation*}
Then,
\begin{equation*}
\tau(\varphi)  = \tau(\phi) - |\d \varphi|^2\varphi = \tau(\phi) - \frac{1}{2}(k_1^2+l_1^2+k_2^2+l_2^2)\varphi,
\end{equation*}
and thus
\begin{align*}
\tau(\phi) =& -\frac{1}{\sqrt{2}}\left(\frac{1}{2}(l_1^2+l_2^2-k_1^2-k_2^2)\cos(k_1u+k_2v),\frac{1}{2}(l_1^2+l_2^2-k_1^2-k_2^2)\sin(k_1u+k_2v),\right. \\
            & \left.\frac{1}{2}(k_1^2+k_2^2-l_1^2-l_2^2)\cos(l_1u+l_2v),\frac{1}{2}(k_1^2+k_2^2-l_1^2-l_2^2)\sin(l_1u+l_2v)\right)\\
           =& -\frac{l_1^2+l_2^2-k_1^2-k_2^2}{2\sqrt{2}}\left(\cos(k_1u+k_2v),\sin(k_1u+k_2v),\right. \\
            & \left.-\cos(l_1u+l_2v),-\sin(l_1u+l_2v)\right).
\end{align*}
It follows
$$
|\tau(\phi)|^2 = \frac{\left(l_1^2+l_2^2-k_1^2-k_2^2\right)^2}{4}.
$$
Thus, $\tau(\phi)=0$ if and only if
\begin{equation}\label{eq8}
l_1^2+l_2^2 = k_1^2+k_2^2.
\end{equation}
By direct computations, we obtain
\begin{align*}
\Delta^2\varphi =& \frac{1}{\sqrt2}\left((k_1^2+k_2^2)^2\cos(k_1u+k_2v),(k_1^2+k_2^2)^2\sin(k_1u+k_2v),\right. \\
                 & \left.(l_1^2+l_2^2)^2\cos(l_1u+l_2v),(l_1^2+l_2^2)^2\sin(l_1u+l_2v)\right).
\end{align*}
Then
\begin{align*}
\textnormal{div }\theta^\sharp =& \partial_u\left\langle\varphi_u,-\Delta\varphi\right\rangle + \partial_v\left\langle\varphi_v,-\Delta\varphi\right\rangle\\
                               =& \partial_u\left\langle\varphi_u,\tau(\phi)-|\d \varphi|^2\varphi\right\rangle + \partial_v\left\langle\varphi_v,\tau(\phi)-|\d \varphi|^2\varphi\right\rangle\\
                               =& \partial_u\left(\left\langle\varphi_u,\tau(\phi)\right\rangle - |\d \varphi|^2\left\langle\varphi_u,\varphi\right\rangle\right) + \partial_v\left(\left\langle\varphi_v,\tau(\phi)\right\rangle - |\d \varphi|^2\left\langle\varphi_v,\varphi\right\rangle\right)\\
                               =& 0.
\end{align*}
Next, by replacing in the left hand side of Equation \eqref{eq1} and by direct computation we obtain
\begin{align*}
&\frac{1}{\sqrt2}\left( (k_1^2+k_2^2)^2\cos(k_1u+k_2v),(k_1^2+k_2^2)^2\sin(k_1u+k_2v),\right. \\
&\left.(l_1^2+l_2^2)^2\cos(l_1u+l_2v),(l_1^2+l_2^2)^2\sin(l_1u+l_2v) \right) \\
&- \frac{k_1^2+l_1^2+k_2^2+l_2^2}{\sqrt 2}\left((k_1^2+k_2^2)\cos(k_1u+k_2v),(k_1^2+k_2^2)\sin(k_1u+k_2v),\right.\\
&\left.(l_1^2+l_2^2)\cos(l_1u+l_2v),(l_1^2+l_2^2)\sin(l_1u+l_2v)\right) \\
&+ \left(-\frac{1}{2}\left((k_1^2+k_2^2)^2+(l_1^2+l_2^2)^2\right) + \frac{1}{2}(k_1^2+l_1^2+k_2^2+l_2^2)^2\right)\cdot\\ &\cdot\frac{1}{\sqrt2}\left(\cos(k_1u+k_2v),\sin(k_1u+k_2v),\cos(l_1u+l_2v),\sin(l_1u+l_2v)\right) = 0.
\end{align*}
Therefore the map $\phi$ is biharmonic. Taking into account \eqref{eq8}, we conclude that $l_1^2+l_2^2 \neq k_1^2+k_2^2$ if and only if $\phi$ is proper biharmonic.
\end{proof}
\begin{remark}
We note that if $k_1l_2 - k_2l_1 = \pm 1$ and $l_1^2+l_2^2 \neq k_1^2+k_2^2$, then $G$ is a diffeomorphism and $\phi$ is proper biharmonic..
\end{remark}

\bigskip

\section{Biharmonic homogeneous polynomial maps}

Next, we give another application of Theorem \ref{t24} for a particular class of maps. Consider the diagram below
\bigskip
\begin{center}
  \begin{tikzpicture}
  \matrix (m) [matrix of math nodes,row sep=3em,column sep=4em,minimum width=2em]
  {
     \r^{m+1} & \r^{n+1} \\
     \s^m & \s^n \\};
  \path[-stealth]
    (m-2-1) edge node [left] {$\i$} (m-1-1)
    (m-1-1) edge node [above] {$F$} (m-1-2)
    (m-2-1) edge node [above] {$\varphi$} (m-1-2)
    (m-2-2) edge node [right] {$\i$} (m-1-2)
    (m-2-1) edge node [below] {$\phi$} (m-2-2);
\end{tikzpicture}
\end{center}
where $F:\r ^{m+1}\rightarrow\r ^{n+1}$ is a vector valued function such that each component is a homogeneous polynomial of degree $k$. Such a map $F$ is called form of degree $k$. When $k=2$, we say that $F$ is a quadratic form, and we will keep the same terminology also for the induced map $\phi$. We will always assume that $\phi$ is not constant.
\begin{theorem}\label{t3}
The bitension field of the map $\phi$ is given by
\begin{align}\label{eq9}
\tau_2(\phi) = & \Dz\Dz F + 2\left(mk+2k^2-3k-m+3-\left|\dz F\right|^2\right)\Dz F \nonumber \\
               & + \left( - 2\Dz\left(\left|\dz F\right|^2\right) - 2\left|\nz\dz F\right|^2 + \left|\Dz F\right|^2 \right.\\
               & \left. - 2\left(2mk+6k^2-6k-m+3\right)\left|\dz F\right|^2 + 2\left|\dz F\right|^4 + 4k^2(m+2k-1)\right) \varphi \nonumber \\
               & + 2\dz F\left(\stackrel{o}{\textnormal{grad}}\left(\left|\dz F\right|^2\right)\right), \nonumber
\end{align}
where $\dz$, $\nz$ and $\Dz$ denote operators that act on $\r^{m+1}$.
\end{theorem}

\begin{proof}
For simplicity of notation, an arbitrary point $p\in\s^m$ will be denoted by $\overline x$. Now we consider an arbitrary point $\overline x\in\s^m$ and a geodesic frame field $\{X_i\}_{i=1}^m$ around $\overline x$, defined on the open subset $U$ of $\s^m$, $\overline x\in U$.

On U, we have:
\begin{dmath}\label{eq10}
|\d \varphi|^2=\sum_{i=1}^{m}\left|\d \varphi(X_i)\right|^2 = \sum_{i=1}^{m}\left|\dz F(X_i)\right|^2 = \left|\dz F\right|^2-\left|\dz F(r)\right|^2 = \left|\dz F\right|^2-| r F|^2
\end{dmath}
Then, at $\overline x$ we have
\begin{dmath*}
|\d \varphi|^2_{\overline x} = \left|\dz F\right|^2_{\overline x} - \left|(F\circ\gamma)'(1)\right|^2,
\end{dmath*}
where $\gamma(t)= t\overline x = t r(\overline x)$.

On $U$ we have
\begin{dmath}\label{eq11}
\tau(\varphi) = \dz F(\tau(i))+\textnormal{trace}\nz\dz F(\d \i\cdot,\d \i\cdot) = \dz F(-m r)+\tz(F)-(\nz\dz F)(r,r) = - m r F + \tz(F) -  r(r F)+ r F = \tz(F)+(1-m) r F- r(r F).
\end{dmath}
Equivalently, on $U$:
\begin{dmath*}
\Delta\varphi = \Dz F - (1-m)r F + r(r F).
\end{dmath*}

At $\overline x$, for $(\nz\dz F)_{\overline x}( r, r)$ we can give an alternative formula. Indeed, let $\gamma(t) = t \overline x,$ $t\in\r $. Then, $\gamma$ is a geodesic in $\r ^{m+1}$ and
\begin{dmath*}
\left(\nz \dz F\right)_{\overline x}(r,r) = \left(\nz \dz F\right)_{\overline x}\left(\gamma'(1),\gamma'(1)\right) = \gamma'(1)\left\{\gamma'(t)F\right\} - \left(\nz_{\gamma'(1)}\gamma'(t)\right)(F) = \frac{\d ^2}{\d t^2}\Big|_{t=1}\left\{(F\circ\gamma)(t)\right\}.
\end{dmath*}
So, at $\overline x$:
\begin{dmath*}
\tau(\varphi)_{\overline x} = \tz(F)_{\overline x} - m\frac{\d }{\d t}\Big|_{t=1}\left\{F(t\overline x)\right\}-\frac{\d ^2}{\d t^2}\Big|_{t=1}\left\{F(t\overline x)\right\}.
\end{dmath*}

From now on, we will assume that $F$ is a form of degree $k$. Let us prove that the action of the vector field $r$ on $F$ is $rF = kF$, on $\r ^{m+1}$. For $\overline x_0 = 0$ the equality is obvious. Then, at $\overline x\in\r ^{m+1}\backslash\{0\}$ we have
\begin{align*}
(r F)(\overline x) = &  r(\overline x)F \\
                   = & \frac{\d }{\d t}\Big|_{t=1}\left\{F(t\overline x)\right\} = \frac{\d}{\d t}\Big|_{t=1} \left\{t^kF(\overline x)\right\} \\
                   = & kF(\overline x).
\end{align*}
Moreover, this kind of formula holds for the action of $r$ on $\Dz F$, even if $\Dz F$ is not a form of degree $k-2$. Indeed, each component of $\Dz F$ is either a homogeneous polynomial of degree $k-2$ or $0$ and we have
\begin{equation*}
\left(\Dz F\right)(t\overline x) = t^{k-2}\left(\Dz F\right)(\overline x), \quad \forall \overline x\in\r ^{m+1}.
\end{equation*}
Thus,
\begin{equation}\label{eq12}
r\left(\Dz F\right) = (k-2)\Dz F.
\end{equation}

On $\s^m$, from the eqaution \eqref{eq11} we have the following
\begin{dmath}\label{eq13}
\tau(\varphi) = \tz(F) + (1-m)k\varphi - k^2\varphi = \tz(F) - k(m+k-1)\varphi.
\end{dmath}
Thus, on $\s^m$:
\begin{equation}\label{eq14}
\Delta\varphi = \Dz F + k(m+k-1)\varphi.
\end{equation}
On $\s^m$, the Equation \eqref{eq10} becomes
\begin{equation}\label{eq15}
|\d \varphi|^2 = \left|\dz F\right|^2 - k^2.
\end{equation}
Next, on $\s^m$ we have:
\begin{dmath*}
\textnormal{grad}\left(|\d \varphi|^2\right) = \textnormal{grad}\left(\left|\dz F\right|^2\right) = \stackrel{o}{\textnormal{grad}}\left(\left|\dz F\right|^2\right) -  r\left(\left|\dz F\right|^2\right) r.
\end{dmath*}
So, on $\s^m$:
\begin{align*}
2\d \varphi\left(\textnormal{grad}\left(|\d \varphi|^2\right)\right) &= 2\dz F\left(\stackrel{o}{\textnormal{grad}}\left(\left|\dz F\right|^2\right) -  r\left(\left|\dz F\right|^2\right) r\right) \\
                                                                     &= 2\dz F\left(\stackrel{o}{\textnormal{grad}}\left(\left|\dz F\right|^2\right)\right) - 2k\left[ r\left(\left|\dz F\right|^2\right)\right]\varphi.
\end{align*}

Further, we compute the norm $|\tau(\varphi)|^2$. Using Equations \eqref{eq13}, \eqref{eq3} and \eqref{eq15}, we get
\begin{align*}
|\tau(\varphi)|^2 =& \left|\tz(F) - k(m+k-1)\varphi\right|^2 \\
                  =& \left|\tz(F)\right|^2 + k^2(m+k-1)^2 - 2k(m+k-1)\left\langle\tz(F),\varphi\right\rangle\\
                  =& \left|\tz(F)\right|^2 + k^2(m+k-1)^2\\
                   & - 2k(m+k-1)\left\langle\tau(\phi)-|\d \varphi|^2\varphi+k(m+k-1)\varphi,\varphi\right\rangle\\
                  =& \left|\tz(F)\right|^2 + k^2(m+k-1)^2 - 2k(m+k-1)\left(-|\d \varphi|^2+k(m+k-1)\right)\\
                  =& \left|\tz(F)\right|^2 - k^2(m+k-1)^2 - 2k(m+k-1)\left(\left|\dz F\right|^2-k^2\right)\\
                  =& \left|\tz(F)\right|^2 - k^2(m+k-1)^2 - 2k^3(m+k-1) + 2k(m+k-1)\left|\dz F\right|^2\\
                  =& \left|\tz(F)\right|^2 + 2k(m+k-1)\left|\dz F\right|^2 - k(m+k-1)(km+k^2-k+2k^2).
\end{align*}
We conclude that
\begin{equation*}
|\tau(\varphi)|^2 = \left|\tz(F)\right|^2 + 2k(m+k-1)\left|\dz F\right|^2 - k^2(m+k-1)(m+3k-1).
\end{equation*}

From Equations \eqref{eq12} and \eqref{eq14} we get
\begin{align*}
\tau_2(\varphi) =& \Delta\Delta\varphi = \Delta(\Dz F) + k(m+k-1)\Delta\varphi \\
                =& \Dz \Dz F - (1-m) r(\Dz F) + r\left(r(\Dz F)\right)\\
                 & + k(m+k-1)\Dz F + k^2(m+k-1)^2\varphi \\
                =& \Dz \Dz F - (1-m)(k-2)\Dz F + (k-2)^2\Dz F \\
                 & + k(m+k-1)\Dz F + k^2(m+k-1)^2\varphi \\
                =& \Dz \Dz F + 2(mk+k^2-3k-m+3)\Dz F + k^2(m+k-1)^2\varphi.
\end{align*}
From Equation \eqref{eq15} we obtain
\begin{equation}\label{eq16}
\Delta\left(|\d \varphi|^2\right) = \Delta\left(\left|\dz F\right|^2\right) = \Dz\left(\left|\dz F\right|^2\right) - (1-m) r\left(\left|\dz F\right|^2\right) + r\left( r\left(\left|\dz F\right|^2\right)\right).
\end{equation}
Now, at $\overline x$ we have:
\begin{equation*}
\Delta\left(|\d \varphi|^2\right)_{\overline x} = \Dz\left(\left|\dz F\right|^2\right)_{\overline x} + m\frac{\d}{\d t}\Big|_{t=1}\left\{\left|\dz F\right|^2(t\overline x)\right\} + \frac{\d^2}{\d t^2}\Big|_{t=1}\left\{\left|\dz F\right|^2(t\overline x)\right\}.
\end{equation*}
We note that $\left|(\dz F)(t\overline x)\right|^2 = t^{2(k-1)}\left|(\dz F)(\overline x)\right|^2$, for all $\overline x\in\r ^{m+1}$ and for all $t\in\r$, so
$$
r\left(\left|\dz F\right|^2\right) = 2(k-1)\left|\dz F\right|^2
$$
and
$$
r\left( r\left(\left|\dz F\right|^2\right)\right) = 4(k-1)^2\left|\dz F\right|^2.
$$
Thus, Equation \eqref{eq16} becomes
\begin{align*}
\Delta\left(|\d \varphi|^2\right) &= \Dz\left(\left|\dz F\right|^2\right) - 2(1-m)(k-1)\left|\dz F\right|^2 + 4(k-1)^2\left|\dz F\right|^2 \\
                                  &= \Dz\left(\left|\dz F\right|^2\right) + 2(k-1)(m+2k-3)\left|\dz F\right|^2.
\end{align*}

In the following, first we will compute $\theta$, and then we will compute $\textnormal{div}\theta^\sharp$.
\begin{align*}
\theta(X) &= \left\langle\d\phi(X),\tau(\phi)\right\rangle = \left\langle\d\varphi(X),\tau(\varphi)\right\rangle\\
          &= \left\langle\dz F(X_i),\tz(F)-k(m+k-1)\varphi\right\rangle\\
          &= \left\langle\dz F(X_i),\tz(F)\right\rangle, \quad \textnormal{ on } \s^m.
\end{align*}
We know that
$$
\textnormal{div}\theta^\sharp = \sum_{i=1}^{m}\left\langle X_i,\nabla_{X_i}\theta^\sharp\right\rangle.
$$
Then, at $\overline x$ we have
\begin{align*}
\textnormal{div}\theta^\sharp   &= \sum_{i=1}^{m}X_i\left\langle X_i,\theta^\sharp\right\rangle\\
                                &= \sum_{i=1}^{m}X_i\left(\theta(X_i)\right) = \sum_{i=1}^{m}X_i\left\langle\dz F(X_i),\tz(F)\right\rangle\\
                                &= \sum_{i=1}^{m}\left\{\left\langle\nz_{X_i}\dz F(X_i),\tz(F)\right\rangle + \left\langle\dz F(X_i),\nz_{X_i}\tz(F)\right\rangle\right\}\\
                                &= \sum_{i=1}^{m}\left\{\left\langle\left(\nz\d F\right)(X_i,X_i)+\dz F\left(\nz_{X_i}X_i\right),\tz(F)\right\rangle+\left\langle\dz F(X_i),\nz_{X_i}\tz(F)\right\rangle\right\}\\
                                &= \sum_{i=1}^{m}\left\{\left\langle\left(\nz\d F\right)(X_i.X_i)- r F,\tz(F)\right\rangle\right\} +\left\langle\dz F,\dz\left(\tz(F)\right)\right\rangle - \left\langle r(F), r\left(\tz(F)\right)\right\rangle.
\end{align*}
Further
\begin{align*}
\textnormal{div}\theta^\sharp   &= \left\langle\Dz F + r(rF) - rF + mrF,\Dz F\right\rangle + \left\langle\dz F,\dz\left(\tz(F)\right)\right\rangle + k(k-2)\left\langle \varphi,\Dz F\right\rangle\\
                                &= \left\langle\Dz F + k(m+k-1)\varphi,\Dz F\right\rangle - \left\langle\dz F, \dz\left(\Dz F)\right)\right\rangle + k(k-2)\left\langle \varphi,\Dz F\right\rangle\\
                                &= \left|\Dz F\right|^2 + k(m+k-1)\left\langle \varphi,\Dz F\right\rangle - \left\langle \dz F, \dz\left(\Dz F\right)\right\rangle + k(k-2)\left\langle \varphi,\Dz F\right\rangle\\
                                &= \left|\Dz F\right|^2 + k(m+2k-3)\left\langle \varphi,\Dz F\right\rangle - \left\langle \dz F, \dz\left(\Dz F\right)\right\rangle\\
                                &= \left|\Dz F\right|^2 - k(m+2k-3)\left[-\left|\dz F\right|^2+k^2+k(m+k-1)\right] - \left\langle \dz F, \dz\left(\Dz F\right)\right\rangle\\
                                &= \left|\Dz F\right|^2 + k(m+2k-3)\left|\dz F\right|^2-k^2(m+2k-3)(m+2k-1) - \left\langle \dz F, \dz\left(\Dz F\right)\right\rangle.
\end{align*}
Now, we want to find an expression for $\left\langle \dz F, \dz\left(\Dz F\right)\right\rangle$.
We know that $\dz \Dz = \Dz \dz$. Then, on $\r^{m+1}$,
\begin{equation*}
\left\langle\dz F,\dz(\Dz F)\right\rangle = \left\langle\dz F,\Dz(\dz F)\right\rangle.
\end{equation*}
From the Weitzenb\"{o}k formula for $\dz F$ we get
\begin{equation*}
\frac{1}{2}\Dz\left(\left|\dz F\right|^2\right) = \left\langle\dz F,\Dz\dz F\right\rangle - \left|\nz\dz F\right|^2.
\end{equation*}
Then, on $\s^m$ we have
\begin{dmath*}
\textnormal{div}\theta^\sharp = \left|\dz F\right|^2 + k(m+2k-3)\left|\dz F\right|^2 -k^2(m+2k-3)(m+2k-1) - \frac{1}{2}\Dz\left(\left|\dz F\right|^2\right) - \left|\nz\dz F\right|^2.
\end{dmath*}
Recall that
\begin{dmath*}
2\d\varphi\left(\textnormal{grad}\left(|\d\varphi|^2\right)\right) = 2\dz F\left(\stackrel{o}{\textnormal{grad}}\left(\left|\dz F\right|^2\right)\right) - 2k  r\left(\left|\dz F\right|^2\right)\varphi.
\end{dmath*}
The last term can be rewritten and we have, on $\s^m$
\begin{dmath}\label{eq17}
2\d\varphi\left(\textnormal{grad}\left(|\d\varphi|^2\right)\right) = 2\dz F\left(\stackrel{o}{\textnormal{grad}}\left(\left|\dz F\right|^2\right)\right) - 4k(k-1)\left|\dz F\right|^2\varphi.
\end{dmath}

Next, we replace in Equation \eqref{ec3} the terms we have computed:
\begin{align*}
\tau_2(\phi) =& \Dz\Dz F + 2(mk+k^2-3k-m+3)\Dz F + k^2(m+k-1)^2\varphi \\
              & + 2\left(\left|\dz F\right|^2-k^2\right)\left(-\Dz F-k(m+k-1)\varphi\right)\\
              & + \left(-\Dz\left(\left|\dz F\right|^2\right) - 2(k-1)(m+2k-3)\left|\dz F\right|^2 - \Dz\left(\left|\dz F\right|^2\right) - 2\left|\nz\dz F\right|^2 \right. \\
              & \left. + 2\left|\Dz F\right|^2 + 2k(m+2k-3)\left|\dz F\right|^2 - 2k^2(m+2k-3)(m+2k-1)\right. \\
              & \left. - \left|\Dz F\right|^2 - 2k(m+k-1)\left|\dz F\right|^2 + k^2(m+k-1)(m+3k-1) + 2\left|\dz F\right|^4 + 2k^4 \right. \\
              & \left. - 4\left|\dz F\right|^2k^2\right)\varphi + 2\dz F\left(\stackrel{o}{\textnormal{grad}}\left(\left|\dz F\right|^2\right)\right) - 4k(k-1)\left|\dz F\right|^2\varphi.
\end{align*}
From the above relation,
\begin{align*}
\tau_2(\phi) = & \Dz\Dz F + 2\left(mk+2k^2-3k-m+3-\left|\dz F\right|^2\right)\Dz F \\
               & + \left( - 2\Dz\left(\left|\dz F\right|^2\right) - 2\left|\nz\dz F\right|^2 + \left|\Dz F\right|^2 \right.\\
               & \left. - 2\left(2mk+6k^2-6k-m+3\right)\left|\dz F\right|^2 + 2\left|\dz F\right|^4 + 4k^2(m+2k-1)\right) \varphi\\
               & + 2\dz F\left(\stackrel{o}{\textnormal{grad}}\left(\left|\dz F\right|^2\right)\right).
\end{align*}
\end{proof}

\begin{remark}\label{r1}
We note that, using Equations \eqref{eq3}, \eqref{eq13} we obtain that
\begin{align*}
  \tau(\phi) =& -\Dz F - k(m+k-1)\varphi + \left(\left|\dz F\right|^2-k^2\right) \varphi \\
             =& -\Dz F + \left(\left|\dz F\right|^2 - k(m+2k-1)\right) \varphi.
\end{align*}
If $\Dz F = 0$, since $\varphi$ is normal to $\s^m$ and $\tau(\phi)$ is tangent, from the above relation we conclude that
$$
\tau(\phi) = 0
$$
and, on $\s^m$,
$$
\left|\dz F\right|^2 - k(m+2k-1) = 0,
$$
so the energy density $e(\phi)$ is constant. Moreover, when $k = 2$ we can see that, since $\phi$ is not constant, $\Dz F = 0$ if and only if $\tau(\phi) = 0$. We look for proper biharmonic maps, thus we will consider only the case $\Dz F \neq 0$. For more details about harmonic homogeneous polynomial maps of degree $k$, see \cite{BW03} and \cite{ER93}.
\end{remark}

\section{Biharmonic homogeneous polynomials of degree 2}

Let $F:\r ^{m+1}\rightarrow\r ^{n+1}$ be a quadratic form. Then, $F$ can be written in the form
$$
F(\overline{x}) = \left(X^t A_1 X, \ldots ,X^t A_{n+1}\ X\right),
$$
where $\overline x = \left(x^1, x^2, \ldots , x^{m+1}\right)$ coresponds to $X^t = \left[x^1 \ x^2 \ \ldots \ x^{m+1}\right]$, and $A_1$, ..., $A_{n+1}$ are square matrices of order $m+1$, such that if $|\overline x| = 1$, then $|F(\overline x)| = 1$. We note that, by a standard symmetrization process, we can always assume that $A_i$ is symmetric, $i = 1, 2, \ldots, n+1$.

We will always assume that $\phi$ is not a constant map, therefore there exist $i_0\in\{1,2,\ldots,n+1\}$ such that $A_{i_0}$ is not $I_{m+1}$ multiplied by a non-zero real constant.

In this case, we will compute all the terms from the relation \eqref{eq9}. In the following, for the matrices we will use the Frobenius Norm, which is defined as the square root of the sum of the squares of the elements of the matrix.

We immediately observe that since $F$ is a quadratic form, then $\Dz\Dz F = 0$, on $\r^{m+1}$. Then,
\begin{align}
    DF(\overline x) &= \begin{bmatrix}
                        2 X^t A_1 \\
                        2 X^t A_2 \\
                        \vdots \\
                        2 X^t A_{n+1}
                       \end{bmatrix}
\end{align}
As an immediate result, we obtain that on $\r^{m+1}$
\begin{equation}\label{eq18}
\left|\dz F(\overline x)\right|^2 = 4 X^t\left(A_1^2 + A_2^2 + \dots + A_{n+1}^2\right) X = 4 X^t S  X,
\end{equation}
where we denoted $S = A_1^2 + A_2^2 + \dots + A_{n+1}^2$.

Then, we easily obtain that on $\r^{m+1}$
\begin{align}\label{eq19}
  &\Dz F = -\left(2\textnormal{tr}A_1,  2\textnormal{tr}A_2, \ldots , 2\textnormal{tr}A_{n+1}\right), \nonumber\\
  &\Dz\left(\left|\dz F\right|^2\right) = -8\textnormal{tr}\left(A_1^2 + A_2^2 + \dots + A_{n+1}^2\right) = -8 \textnormal{tr}S, \nonumber\\
  &\left|\nz\dz F\right|^2 = 4\left(|A_1|^2 + \dots + |A_{n+1}|^2\right),\\
  &\stackrel{o}{\textnormal{grad}}\left(\left|\dz F\right|^2\right) = 8X^t\left(A_1^2 + A_2^2 + \dots + A_{n+1}^2\right) = 8X^tSX, \nonumber\\
  &\dz F\left(\stackrel{o}{\textnormal{grad}}\left(\left|\dz F\right|^2\right)\right) = 16 \left(X^t A_1S X, \ldots ,X^t A_{n+1}S X\right). \nonumber
\end{align}

We observe that, since the matrices $A_1$, ..., $A_{n+1}$ are symmetric, then
$$
|A_1|^2 + \dots + |A_{n+1}|^2 = \textnormal{tr}S.
$$

Now, we can state the main result of this section.
\begin{theorem}\label{t4}
Let $F:\r ^{m+1}\rightarrow\r ^{n+1}$ be a quadratic form given by
$$
F(\overline{x}) = \left(X^t A_1 X, \ldots ,X^t A_{n+1}\ X\right),
$$
such that if $|\overline x| = 1$ then $|F(\overline x)| = 1$.  We consider $\phi:\s^m\rightarrow\s^n$ defined by $\phi(\overline x) = F(\overline x)$ and $\varphi = \i\circ\phi:\s^m\rightarrow\r^{n+1}$. If we denote $S = A_1^2 + A_2^2 + \dots + A_{n+1}^2$ then, at a point $\overline x\in\s^m$, the bitension field of $\phi$ has the following expression
\begin{align}\label{eq20}
\tau_2(\phi)  =& -4\left(m + 5 - 4X^tSX\right)\left(\textnormal{tr}A_1, \textnormal{tr}A_2, \ldots , \textnormal{tr}A_{n+1}\right) \nonumber\\
               & + \left(8\textnormal{tr}S + 4\left( \left(\textnormal{tr}A_1\right)^2 + \dots + \left(\textnormal{tr}A_{n+1}\right)^2 \right) \right. \\
               & \left. - 8(3m + 15)X^tSX + 32 \left(X^tSX\right)^2 + 16(m + 3)\right) \varphi \nonumber\\
               & + 32 \left(X^t A_1S X, \ldots ,X^t A_{n+1}S X\right).\nonumber
\end{align}
\end{theorem}
\begin{proof}
The proof follows immediately from Equations \eqref{eq9}, for $k = 2$ and using Equations \eqref{eq18} and \eqref{eq19}.
\end{proof}

We note that, using Equations \eqref{eq15} and \eqref{eq18}, the condition $S = \alpha I_{m+1}$, for some real constant $\alpha$, is equivalent to $|\d\phi|^2$ is constant.

From Equation \eqref{eq20}, and the fact that $\Dz F$ is constant, we obtain
\begin{proposition}\label{p2}
If the quadratic form $\phi$ has constant energy density, then $\phi$ is proper biharmonic if and only if we have on $\s^m$
\begin{equation}\label{eq22}
\left|\dz F\right|^2 = m + 5 \quad \textnormal{and} \quad \left|\Dz F\right|^2 = 4\left( \left(\textnormal{tr}A_1\right)^2 + \dots + \left(\textnormal{tr}A_{n+1}\right)^2 \right) = 2(m+1)^2.
\end{equation}
\end{proposition}
\begin{proof}
Since the map $\phi$ has constant energy density, it follows that $S = \alpha I_{m+1}$, $\alpha \neq 0$. Using Equation \eqref{eq20}, we immediately obtain
\begin{dmath}\label{eq23}
\tau_2(\phi) = -4\left(m + 5 -4\alpha\right)\left(\textnormal{tr}A_1, \textnormal{tr}A_2, \ldots , \textnormal{tr}A_{n+1}\right) + \left( 4\left( \left(\textnormal{tr}A_1\right)^2 + \dots + \left(\textnormal{tr}A_{n+1}\right)^2 \right) + 32 \alpha^2 - 8(2m + 10) \alpha + 16(m+3) \right) \varphi = 2\left(m + 5 -4\alpha\right) \Dz F + \left(\left|\Dz F\right|^2  + 32 \alpha^2 - 8(2m + 10) \alpha + 16(m+3) \right) \varphi
\end{dmath}

Recall that there exist $i_0\in\{1,2,\ldots,n+1\}$ such that $A_{i_0}$ is not $I_{m+1}$ multiplied by e real non-zero real constant. If $\tau_2(\phi) = 0$, then looking at the $i_0$-th component we get
$$
\left|\Dz F\right|^2  + 32 \alpha^2 - 8(2m + 10) \alpha + 16(m+3) = 0.
$$
But this implies, since $\Dz F \neq 0$, that
$$
m + 5 - 4\alpha = 0.
$$
Thus
$$
\tau_2(\phi) = 0 \Longleftrightarrow m + 5 - 4\alpha = 0 \textnormal{ and } \left|\Dz F\right|^2 = 2(m + 1)^2.
$$
Moreover, such a map can not be harmonic.
\end{proof}

\begin{remark}
If $T$ is an orthogonal matrix, i.e. $T^t\cdot T = I_{m+1}$, then $F\circ T$ is given by the symmetric matrices
$$
A_1' = T^tA_1T, \ldots,A_{n+1}' = T^tA_{n+1}T.
$$
Thus
$$
\left(A_1'\right)^2 + \dots + \left(A_{n+1}'\right)^2 = T^tST = \alpha I_{m+1}
$$
and, since $\textnormal{tr}(T^t A_i T) = \textnormal{tr}A_i$, we have
$$
\left(\textnormal{tr}A_1'\right)^2 + \dots + \left(\textnormal{tr}A_{n+1}'\right)^2 = \left(\textnormal{tr}A_1\right)^2 + \dots + \left(\textnormal{tr}A_{n+1}\right)^2.
$$
\end{remark}

\begin{example}
We consider $F:\mathbb{C}^4\rightarrow\r ^{10}$ given by
\begin{align*}
F\left(z^1, z^2, z^3, z^4\right) =& \left(\left|z^1\right|^2 + \left|z^2\right|^2, z^1z^3 - z^2z^4, z^1z^3 + z^2z^4,\right.\\
                       & \left. z^2z^3 - z^1z^4, z^2z^3 + z^1z^4, \left|z^3\right|^2 + \left|z^4\right|^2\right)
\end{align*}
We can identify $F$ with the quadratic map $F:\r^8\rightarrow\r^{10}$
\begin{align*}
F&\left(x^1,y^1,x^2,y^2,x^3,y^3,x^4,y^4\right) = \\
                 =& \left(\left(x^1\right)^2 + \left(y^1\right)^2 + \left(x^2\right)^2 + \left(y^2\right)^2, x^1x^3 - x^2x^4 - y^1y^3 + y^2y^4,\right.\\
                  & \left. x^1y^3 + x^3y^1 - x^2y^4 - x^4y^2, x^1x^3 + x^2x^4 - y^1y^3 - y^2y^4, \right.\\
                  & \left. x^1y^3 + x^3y^1 + x^2y^4 + x^4y^2, x^2x^3 - x^1x^4 + y^1y^4 - y^2y^3, \right.\\
                  & \left. x^2y^3 + x^3y^2 - x^1y^4 - x^4y^1, x^2x^3 + x^1x^4 - y^1y^4 - y^2y^3, \right.\\
                  & \left.x^2y^3 + x^3y^2 + x^1y^4 + x^4y^1, \left(x^3\right)^2 + \left(y^3\right)^2 + \left(x^4\right)^2 + \left(y^4\right)^2 \right)
\end{align*}

It is clear that $F$ maps $\s^7$ into $\s^9$. The map F is given by the matrices
\begin{align*}
  A_1 = \begin{bmatrix}
             I_4 & O_4 \\
             O_4 & I_4
           \end{bmatrix},
\quad
  A_2 = \frac{1}{2}\begin{bmatrix}
             O_4 & B_2 \\
             B_2 & O_4
             \end{bmatrix},
\quad
  A_3 = \frac{1}{2}\begin{bmatrix}
            O_4 & B_3 \\
            B_3 & O_4
           \end{bmatrix},
\quad
  A_4 = \frac{1}{2}\begin{bmatrix}
            O_4 & B_4 \\
            B_4 & O_4
           \end{bmatrix},
\end{align*}
\begin{align*}
  A_5 = \frac{1}{2}\begin{bmatrix}
             O_4 & B_5 \\
             B_5 & O_4
           \end{bmatrix},
\quad
  A_6 = \frac{1}{2}\begin{bmatrix}
            O_4 & B_6 \\
            B_6^t & O_4
           \end{bmatrix},
\quad
  A_7 = \frac{1}{2}\begin{bmatrix}
            O_4 & B_7 \\
            B_7^t & O_4
           \end{bmatrix},
\quad
  A_8 = \frac{1}{2}\begin{bmatrix}
            O_4 & B_8 \\
            B_8 & O_4
           \end{bmatrix},
\end{align*}
\begin{align*}
  A_9 = \frac{1}{2}\begin{bmatrix}
             O_4 & B_9\\
             B_9 & O_4
           \end{bmatrix},
\quad
  A_{10} = \begin{bmatrix}
             O_4 & O_4\\
             O_4 & I_4
           \end{bmatrix},
\end{align*}
where
\begin{align*}
B_2 = \begin{bmatrix}
         1 & 0 & 0 & 0 \\
         0 & -1 & 0 & 0 \\
         0 & 0 & -1 & 0 \\
         0 & 0 & 0 & 1
      \end{bmatrix},
\quad
B_3 = \begin{bmatrix}
         0 & 1 & 0 & 0 \\
         1 & 0 & 0 & 0 \\
         0 & 0 & 0 & -1 \\
         0 & 0 & -1 & 0
       \end{bmatrix},
\quad
B_4 = \begin{bmatrix}
         1 & 0 & 0 & 0 \\
         0 & -1 & 0 & 0 \\
         0 & 0 & 1 & 0 \\
         0 & 0 & 0 & -1
      \end{bmatrix},
\end{align*}
\begin{align*}
B_5 = \begin{bmatrix}
         0 & 1 & 0 & 0 \\
         1 & 0 & 0 & 0 \\
         0 & 0 & 0 & 1 \\
         0 & 0 & 1 & 0
      \end{bmatrix},
\quad
B_6 = \begin{bmatrix}
          0 & 0 & -1 & 0 \\
          0 & 0 & 0 & 1 \\
          1 & 0 & 0 & 0 \\
          0 & -1 & 0 & 0
      \end{bmatrix},
\quad
B_7 = \begin{bmatrix}
           0 & 0 & 0 & -1 \\
           0 & 0 & -1 & 0 \\
           0 & 1 & 0 & 0 \\
           1 & 0 & 0 & 0
      \end{bmatrix},
\end{align*}
\begin{align*}
B_8 = \begin{bmatrix}
           0 & 0 & 1 & 0 \\
           0 & 0 & 0 & -1 \\
           1 & 0 & 0 & 0 \\
           0 & -1 & 0 & 0
       \end{bmatrix},
\quad
B_9 = \begin{bmatrix}
           0 & 0 & 0 & 1 \\
           0 & 0 & 1 & 0 \\
           0 & 1 & 0 & 0 \\
           1 & 0 & 0 & 0
      \end{bmatrix}.
\end{align*}
We can easily observe that
$$
A_1^2 + A_2^2 + A_3^2 + A_4^2 + A_5^2 + A_6^2 + A_7^2 + A_8^2 + A_9^2 + A_{10}^2 = 3 I_8,
$$
and thus $F$ satisfies the conditions for Proposition \ref{p2}. In this case, $m = 7$ and $\alpha = 3$. Since $\Dz F \neq 0,$ $m+5 = 4\alpha$ and $\left|\Dz F\right|^2 = 2(m + 1)^2$ we conclude that $F$ is proper biharmonic.

We note that if we act on $F$ with the isometry
\begin{align*}
T = \begin{bmatrix}
    \frac{1}{\sqrt{2}} & 0 & 0 & 0 & 0 & 0 & 0 & 0 & 0 & -\frac{1}{\sqrt{2}} \\
    0 & 1 & 0 & 0 & 0 & 0 & 0 & 0 & 0 & 0 \\
    0 & 0 & 1 & 0 & 0 & 0 & 0 & 0 & 0 & 0 \\
    0 & 0 & 0 & 1 & 0 & 0 & 0 & 0 & 0 & 0 \\
    0 & 0 & 0 & 0 & 1 & 0 & 0 & 0 & 0 & 0 \\
    0 & 0 & 0 & 0 & 0 & 1 & 0 & 0 & 0 & 0 \\
    0 & 0 & 0 & 0 & 0 & 0 & 1 & 0 & 0 & 0 \\
    0 & 0 & 0 & 0 & 0 & 0 & 0 & 1 & 0 & 0 \\
    0 & 0 & 0 & 0 & 0 & 0 & 0 & 0 & 1 & 0 \\
    \frac{1}{\sqrt{2}} & 0 & 0 & 0 & 0 & 0 & 0 & 0 & 0 & \frac{1}{\sqrt{2}} \\
    \end{bmatrix},
\end{align*}
the last component of the $\r^{10}$-valued quadratic form $T\circ F$ is the squared norm of the argument divided by $\sqrt{2}$. Then we remove $\left(T\circ F\right)^{10}$, we multiply the new map with $\sqrt{2}$ and we obtain a new quadratic map $G:\s^7\rightarrow\s^8$ which is harmonic and $|\d G|^2$ is constant. Therefore, this example can be seen as another application to Theorem \ref{t1}.
\end{example}

\begin{proposition}
Let $P:\r^n\times\r^n\rightarrow\r^n$, $n = 2$, $4$, $8$ be the orthogonal multiplication of complex numbers, quaternions or octonions. Then, the quadratic form $\phi_\lambda:\s^{2n-1}\rightarrow\s^{n+1}$, obtained from the restriction of $F_\lambda:\r^{n}\times\r^n\rightarrow\r^{n+2}$ given by
$$
F_\lambda(\overline z,\overline w) = \left(|\overline z|^2 + \lambda |\overline w|^2, \sqrt{2(1-\lambda)}P(\overline z,\overline w), \sqrt{1 - \lambda^2}|\overline w|^2\right),
$$
is a proper biharmonic map if and only if $\lambda=0$. In this case, up to a homothety of the domain and/or target sphere, $\phi_0$ is the composition of the harmonic Hopf fibration
$$
H:\s^{2n-1}\rightarrow\s^n\left(\frac{1}{\sqrt{2}}\right)
$$
followed by the biharmonic inclusion
$$
\i:\s^n\left(\frac{1}{\sqrt{2}}\right)\rightarrow\s^{n+1}.
$$
\end{proposition}
\begin{proof}
The proofs for each case are similar and we will present here only the case $n = 4$. Moreover, the case $n = 2$ will be studied in a more general context in the next Theorem \ref{t4.6}.

Next, we will prove the result for $n = 4$. Let
$$
F_\lambda:\r^8\rightarrow\r^6,
$$
\begin{align*}
F_\lambda&\left(x^1,x^2,x^3,x^4, y^1,y^2,y^3,y^4\right) = \\
         =& \left(\left(x^1\right)^2 + \left(x^2\right)^2 + \left(x^3\right)^2 + \left(x^4\right)^2 + \lambda\left(\left(y^1\right)^2 + \left(y^2\right)^2 + \left(y^3\right)^2 + \left(y^4\right)^2\right),\right. \\
          & \left.\sqrt{2(1 - \lambda)}\left(x^1y^1 - x^2y^2 - x^3y^3 - x^4y^4\right),  \sqrt{2(1 - \lambda)}\left(x^1y^2 + x_2y_1 +x_3y^4 - x^4y^3\right),\right. \\
          & \left. \sqrt{2(1 - \lambda)}\left(x_1y^3-x^2y^4 + x^3y^1 + x^4y^2\right),  \sqrt{2(1 - \lambda)}\left(x^1y^4 + x^2y^3 - x^3y^2 + x^4y^1\right),\right. \\
          & \left. \sqrt{1-\lambda^2}\left(\left(y^1\right)^2 + \left(y^2\right)^2 + \left(y^3\right)^2 + \left(y^4\right)^2\right)\right)
\end{align*}

It is clear that if $|\overline x| = 1$, then $\left|F_\lambda(\overline x)\right| = 1$. The map $F_\lambda$ is given by the matrices

\begin{align*}
  A_1 = \begin{bmatrix}
             I_4 & O_4\\
             O_4 & \lambda I_4
           \end{bmatrix},
\quad
  A_2 = \frac{\sqrt{2(1 - \lambda)}}{2}\begin{bmatrix}
             O_4 & B_2\\
             B_2 & O_4
           \end{bmatrix},
\end{align*}
\begin{align*}
  A_3 = \frac{\sqrt{2(1 - \lambda)}}{2}\begin{bmatrix}
             O_4 & B_3 \\
             B_3^t & O_4
           \end{bmatrix},
\quad
  A_4 = \frac{\sqrt{2(1 - \lambda)}}{2}\begin{bmatrix}
            O_4 & B_4 \\
            B_4^t & O_4
           \end{bmatrix},
\end{align*}
\begin{align*}
  A_5 = \frac{\sqrt{2(1 - \lambda)}}{2}\begin{bmatrix}
            O_4 & B_5 \\
            B_5^t & O_4
           \end{bmatrix},
\quad
  A_6 = \sqrt{1-\lambda^2}\begin{bmatrix}
             O_4 & O_4\\
             O_4 & I_4
           \end{bmatrix},
\end{align*}
where
\begin{align*}
B_2 = \begin{bmatrix}
             1 & 0 & 0 & 0 \\
             0 & -1 & 0 & 0 \\
             0 & 0 & -1 & 0 \\
             0 & 0 & 0 & -1
      \end{bmatrix},
\quad
B_3 = \begin{bmatrix}
             0 & 1 & 0 & 0 \\
             1 & 0 & 0 & 0 \\
             0 & 0 & 0 & 1 \\
             0 & 0 & -1 & 0
      \end{bmatrix},
\end{align*}
\begin{align*}
B_4 = \begin{bmatrix}
             0 & 0 & 1 & 0 \\
             0 & 0 & 0 & -1 \\
             1 & 0 & 0 & 0 \\
             0 & 1 & 0 & 0
      \end{bmatrix},
\quad
B_5 = \begin{bmatrix}
             0 & 0 & 0 & 1 \\
             0 & 0 & 1 & 0 \\
             0 & -1 & 0 & 0 \\
             1 & 0 & 0 & 0
      \end{bmatrix}.
\end{align*}
We can easily observe that
$$
A_1^2 + A_2^2 + A_3^2 + A_4^2 + A_5^2 + A_6^2 + A_7^2 + A_8^2 = (3 - \lambda) I_8,
$$
and thus $F$ satisfies the conditions for Proposition \ref{p2}. In this case, $m = 7$ and $\alpha = 3$. Since $\Dz F_\lambda \neq 0,$ from $m + 5 = 4\alpha$ and $\left|\Dz F_\lambda\right|^2 = 2(m + 1)^2$ we conclude that $\phi_\lambda$ is proper biharmonic if and only if $\lambda = 0$.

The proof for $n = 8$ is similar.

Now, acting with an appropriate isometry on the target manifold, we can see that, up to a homothety of the domain and/or target sphere, the map $\phi_0$ it is the composition of the harmonic Hopf fibration
$$
H:\s^{2n-1}\rightarrow\s^n\left(\frac{1}{\sqrt{2}}\right), \quad \textnormal{ for } n = 2, 4, 8,
$$
followed by the biharmonic inclusion
$$
\i:\s^n\left(\frac{1}{\sqrt{2}}\right)\rightarrow\s^{n+1}.
$$
\end{proof}

\begin{remark}
We note that, for $n = 2$, the map $\phi_0$ is given in \eqref{ec1}.
\end{remark}

\begin{remark}
More examples of proper biharmonic quadratic forms between spheres can be obtained by applying the Theorem \ref{t1} to the harmonic quadratic forms between spheres obtained in \cite{GT87}.
\end{remark}

As a direct consequence of Proposition \ref{p2} we can give the following result:
\begin{proposition}
If $F_1, F_2:\r^{m+1}\rightarrow\r^{n+1}$ are two quadratic forms with constant energy density, such that $\phi_1$ and $\phi_2$ are proper biharmonic, then restriction $\phi$ of the $\r^{2n+2}$-valued quadratic form $F = (\cos\beta F_1, \sin\beta F_2)$, for an arbitrarily fixed $\beta$, is proper biharmonic.
\end{proposition}

In general, when the map $\phi:\s^m\rightarrow\s^n$ is the restriction of a quadratic form $F:\r^{m+1}\rightarrow\r^{n+1}$, because of the homogeneity, $F$ has to bear certain conditions on $\r^{m+1}$ (see \cite{W68}). The next three classification results rely on those properties of $F$.

\begin{theorem}\label{t47}
Let $F:\r^2\rightarrow\r^{n+1}$ be a quadratic form, $n \geq 2$, given by
$$
F(x,y) = \left(x^2, c^1y^2 + 2\gamma^1xy,\ldots, c^ny^2 + 2\gamma^nxy\right),
$$
such that
$$
\left(c^1\right)^2 + \dots + \left(c^n\right)^2 = 1, \quad c^1\gamma^1 + \dots + c^n\gamma^n = 0 \quad \textnormal{and} \quad \left(\gamma^1\right)^2 + \dots + \left(\gamma^n\right)^2 = \frac{1}{2}.
$$
Up to orthogonal transformations of the domain and/or the codomain, the map $\phi:\s^1\rightarrow\s^n$, obtained from the restriction of $F$, is the only proper biharmonic quadratic form from $\s^1$ to $\s^n$. Moreover, the image of $\phi$ is the circle of radius $1/\sqrt{2}$ of $\s^m$.
\end{theorem}
\begin{proof}
From \cite{W68} we know that, by composing with isometries of domain and codomain, any non-constant quadratic form
$$
F = \left(F^1,F^2,\ldots,F^{n+1}\right):\r^2\rightarrow\r^{n+1}
$$
can be transformed such that
$$
F^1(x,y) = x^2 + \mu y^2, \quad -1\leq\mu<1, \ \forall (x,y)\in\r^2
$$
and
$$
\left(F^2,\ldots,F^{n+1}\right) = B + 2G,
$$
where $B$ is an $\r^n$-valued quadratic form in $y$ alone and $G$ is an $\r^n$-valued bilinear form in $x$ and $y$, such that
\begin{equation}\label{w1}
|B|^2 + \mu^2y^4 = y^4, \quad \langle B,G \rangle = 0, \quad \textnormal{and} \quad 2|G|^2 + \mu x^2y^2 = x^2y^2,
\end{equation}
for any $(x,y)\in\r^2$.
Thus, $B$ and $G$ are given by
$$
B = \left(c^1y^2,\ldots,c^ny^2\right) \quad \textnormal{and} \quad G = \left(\gamma^1xy,\ldots,\gamma^nxy\right),
$$
where $c^1, \ldots, c^n$ and $\gamma^1, \ldots, \gamma^n$ are real constants.
With the above notation, from Equation \eqref{w1}, comparing the coefficients of $y^4,$ $xy^3$ and $x^2y^2$, respectively, we obtain
\begin{align}\label{w2}
& \left(c^1\right)^2 + \left(c^2\right)^2 + \dots + \left(c^n\right)^2 = 1 - \mu^2, \nonumber \\
& c^1\gamma^1 + c^2\gamma^2 + \dots + c^n\gamma^n = 0, \\
& \left(\gamma^1\right)^2 + \left(\gamma^2\right)^2 + \dots + \left(\gamma^n\right)^2 = \frac{1-\mu}{2}. \nonumber
\end{align}

The map $F$ is given by the matrices

\begin{align*}
  A_1 = \begin{bmatrix}
             1 & 0 \\
             0 & \mu
         \end{bmatrix},
\quad
  A_{k+1} = \begin{bmatrix}
                 0 & \gamma^k \\
                 \gamma^k & c^k
             \end{bmatrix}, \quad k = 1,2,\ldots,n.
\end{align*}
Thus,
\begin{align}\label{w2.1}
  A_1^2 = \begin{bmatrix}
             1 & 0 \\
             0 & \mu^2
         \end{bmatrix},
\quad
  A_{k+1}^2 = \begin{bmatrix}
                 (\gamma^k)^2 & c^k\gamma^k \\
                 c^k\gamma^k & (c^k)^2 + (\gamma^k)^2
             \end{bmatrix}, \quad k = 1,2,\ldots,n.
\end{align}
Using Equations \eqref{w2} and \eqref{w2.1}, we obtain that
$$
S = \begin{bmatrix}
          \frac{3-\mu}{2} & 0 \\
          0 & \frac{3-\mu}{2}
    \end{bmatrix}.
$$
Since $S = (3-\mu)/2I_4$, from Proposition \ref{p2}, the map $\phi$ is proper biharmonic if and only if
\begin{equation*}
\alpha = \frac{3}{2} \quad \textnormal{and} \quad \left|\Dz F\right|^2 = 4\left( \left(\textnormal{tr}A_1\right)^2 + \dots + \left(\textnormal{tr}A_{n+1}\right)^2 \right) = 8.
\end{equation*}
From the first condition, we immediately see that $\mu$ must be $0$. Then, since
$$
\left(\textnormal{tr}A_1\right)^2 + \dots + \left(\textnormal{tr}A_{n+1}\right)^2 = 2 + 2\mu,
$$
we see that the second condition is automatically satisfied.
Thus, $F$ takes the form
$$
F(x,y) = \left(x^2, c^1y^2 + 2\gamma^1xy, \ldots, c^ny^2 + 2\gamma^nxy\right),
$$
with the conditions \eqref{w2} being satisfied.

If we use the polar coordinates of the unit circle $\s^1$, i.e. $x = \cos\theta$ and $y = \sin\theta$, we obtain the curve
$$
\rho(\theta) = F(\cos\theta, \sin\theta) = \left(\cos^2\theta, c^1\sin^2\theta + \gamma^1\sin(2\theta), \ldots, c^n\sin^2\theta + \gamma^n\sin(2\theta)\right).
$$
By direct computations
$$
\rho'(\theta) = \left(-2\sin(2\theta), c^1\sin(2\theta) + 2\gamma^1\cos(2\theta), \ldots, c^n\sin(2\theta) + 2\gamma^n\cos(2\theta)\right),
$$
and
$$
|\rho'(\theta)| = \sqrt{2}.
$$
If we make the parameter change $s = 2\theta$,  we obtain the curve
\begin{align*}
\tilde\rho(s) =& \left(\frac{1}{2}+\frac{1}{2}\cos s, c^1\left(\frac{1}{2}-\frac{1}{2}\cos s\right) + \gamma^1\sin s, \ldots, c^n\left(\frac{1}{2}-\frac{1}{2}\cos s\right) + \gamma^n\sin s \right) \\
              =& \left(\frac{1}{2}, \frac{c^1}{2}, \dots, \frac{c^n}{2}\right) + \cos s \left(\frac{1}{2}, -\frac{c^1}{2}, \dots, -\frac{c^n}{2}\right) + \sin s\left(0,\gamma^1, \dots, \gamma^n\right) \\
              =& \cos s\cdot \overline f_1 + \sin s\cdot \overline f_2 + \overline f_3,
\end{align*}
where
$$
\overline f_1\perp \overline f_2, \ \overline f_1\perp \overline f_3, \ \overline f_2\perp \overline f_3 \quad \textnormal{and} \quad |\overline f_1|^2 = |\overline f_2|^2 = |\overline f_3|^2 = \frac{1}{2}.
$$
\end{proof}

\begin{theorem}\label{t48}
There are no proper biharmonic quadratic forms from $\s^m$ to $\s^2$, $m\geq2$.
\end{theorem}
\begin{proof}
We split the proof in two main cases.

\textit{Case I:} $m = 2$.

\noindent We can give a short proof of this case using the classification result of the non-constant quadratic forms from $\s^{2n-2}$ to $\s^n$, $n\geq 2$, obtained in \cite{C98}.

Indeed, acording to \cite{C98}, any non-constant quadratic form from $\s^2$ to $\s^2$ is, up to isometries of the domain and codomain, the restriction of the Hopf fribration $\s^3\rightarrow\s^2$ to a totally geodesic $\s^2\subset\s^3$. For example, if we consider $\s^2$ as being the intersection of $\s^3$ and the hyperplane $x^4 = 0$, then $\phi:\s^2\rightarrow\s^2$ is given by
$$
\phi(x^1,x^2,x^3) = \left((x^1)^2 - (x^2)^2 - (x^3)^2, 2x^1x^2, 2x^1x^3\right).
$$
But, it was proved in \cite{WOY14} that such a map is not biharmonic.

However, for the sake of completeness and as an application of Theorem \ref{t4} we present here a more direct proof of our theorem. Let $F:\r^3\rightarrow\r^3$ be a quadratic form such that its restriction $\phi:\s^2\rightarrow\s^2$ is not constant. Using \cite{W68}, we distinguish two subcases:

\emph{Case I.1}
$$
F(x^1, x^2, y) = \left(\left(x^1\right)^2 + \left(x^2\right)^2 + \mu y^2, c^1 y^2 + 2\gamma^1 x^1y + 2\gamma^2 x^2y, c^2 y^2 + 2\gamma^3 x^1y + 2\gamma^4 x^2y\right),
$$
for any $(x^1,x^2,y)\in\r^3$, with $-1\leq \mu <1$,

\emph{Case I.2}
\begin{align*}
F(x,y^1,y^2) = & \left(x^2 + \mu_1\left(y^1\right)^2  + \mu_2\left(y^2\right)^2,\right. \\
               & \left. c^1 \left(y^1\right)^2 + c^2 \left(y^2\right)^2 + 2a^1 y^1y^2 + 2\gamma^1 xy^1 + 2\gamma^2xy^2,\right.  \\
               & \left. c^3 \left(y^1\right)^2 + c^4 \left(y^2\right)^2 + 2a^2 y^1y^2 + 2\gamma^3 xy^1 + 2\gamma^4xy^2\right),
\end{align*}
for any $(x, y^1, y^3)\in\r^3$, with  $-1\leq \mu_1 <1$ and  $-1\leq \mu_2 <1$.

For each subcase we will find all the corresponding maps $F$ and then we will verify the biharmonicity condition from Theorem \ref{t4} or Proposition \ref{p2}.

For \emph{Case I.1}, we denote
$$
\left(F^2,F^3\right) = B + 2G,
$$
where $B$ is an $\r^2$-valued quadratic form in $y$ alone and $G$ is an $\r^2$-valued bilinear form in $\overline x = \left(x^1, x^2\right)$ and $y$, such that
\begin{equation}\label{w3}
|B|^2 + \mu^2y^4 = y^4, \quad \langle B,G \rangle = 0, \quad \textnormal{and} \quad 2|G|^2 + \mu |x|^2y^2 = |x|^2y^2,
\end{equation}
for any $(\overline x,y)\in\r^3$.
Thus, $B$ and $G$ are given by
$$
B = \left(c^1y^2, c^2y^2\right) \quad \textnormal{and} \quad G = \left(\gamma^1x^1y + \gamma^2x^2y,\gamma^3x^1y + \gamma^4x^2y\right),
$$
where $c^1$, $c^2$ and $\gamma^1$, $\gamma^2$, $\gamma^3$, $\gamma^4$ are real constants.
With the above notation, from Equation \eqref{w3}, comparing the coefficients, we obtain
\begin{align}\label{w4}
& \left(c^1\right)^2 + \left(c^2\right)^2 = 1 - \mu^2, & c^1\gamma^1 + c^2\gamma^3 = 0, \nonumber \\
& c^1\gamma^2 + c^2\gamma^4 = 0, & \gamma^1\gamma^2 + \gamma^3\gamma^4 = 0, \\
& \left(\gamma^2\right)^2 + \left(\gamma^4\right)^2 = \frac{1-\mu}{2}, & \left(\gamma^1\right)^2 + \left(\gamma^3\right)^2 = \frac{1-\mu}{2}. \nonumber
\end{align}

If we denote the vectors $\overline v_1 = (c^1, c^2)$, $\overline v_2 = (\gamma^1, \gamma^3)$ and $\overline v_3 = (\gamma^2, \gamma^4)$, by looking at \eqref{w4} we observe that $\overline v_1\perp \overline v_2$, $\overline v_1\perp \overline v_3$  and $\overline v_2\perp \overline v_3$. Since $-1\leq\mu<1$, from the same equations it follows that $\overline v_2 \neq \overline{0}$ and $\overline v_3 \neq \overline{0}$. These imply that $\overline v_1 = \overline{0}$, and further, because $|\overline v_1|^2 = 1 - \mu^2$, we obtain $\mu = -1$.

If we replace $\mu = -1$ in Equations \eqref{w1}, we get
$$
c^1 = c^2 = 0, \quad \gamma^1 = \cos \theta, \ \gamma^2 = -\sin \theta, \ \gamma^3 = \sin \theta \textnormal{ and } \gamma^4 = \cos \theta.
$$
Therefore,
\begin{align*}
F(x^1,x^2,y) = \left((x^1)^2 + (x^2)^2 - y^2, 2\cos\theta x^1y - 2\sin\theta x^2y, 2\sin\theta x^1y + 2\cos\theta x^2y\right).
\end{align*}
The map $\phi$ is nothing but the restriction to $\s^2$ of the Hopf fibration from $\s^3$ onto $\s^2$.

The map $F$ is given by the matrices
\begin{align*}
  A_1 =& \begin{bmatrix}
             1 & 0 & 0\\
             0 & 1 & 0\\
             0 & 0 &-1
         \end{bmatrix},
\quad
  A_2 = \begin{bmatrix}
                 0 & 0 & \cos\theta \\
                 0 & 0 & -\sin\theta \\
                 \cos\theta & -\sin\theta & 0
             \end{bmatrix},
\quad
  A_3 = \begin{bmatrix}
                 0 & 0 & \sin\theta \\
                 0 & 0 & \cos\theta \\
                 \sin\theta & \cos\theta & 0
             \end{bmatrix}.
\end{align*}
Next, since
\begin{align*}
  A_1^2 = \begin{bmatrix}
             1 & 0 & 0\\
             0 & 1 & 0\\
             0 & 0 &\mu^2
         \end{bmatrix},
\quad
  A_2^2 = \begin{bmatrix}
                 \cos^2\theta & -\cos\theta\sin\theta & 0 \\
                 -\cos\theta\sin\theta & \sin^2\theta & 0 \\
                0 & 0 & 1
             \end{bmatrix},
\end{align*}
and
\begin{align*}
  A_3^2 = \begin{bmatrix}
                 \sin^2\theta & \cos\theta\sin\theta & 0 \\
                \cos\theta\sin\theta & \cos^2\theta & 0 \\
                 0 & 0 & 1
             \end{bmatrix},
\end{align*}
using Equations \eqref{w4}, it follows that
\begin{align*}
  S = \begin{bmatrix}
                 2 & 0 & 0 \\
                 0 & 2 & 0 \\
                 0 & 0 & 3
             \end{bmatrix} \quad \textnormal{and}
\quad
  A_1S = \begin{bmatrix}
                 2 & 0 & 0 \\
                 0 & 2 & 0 \\
                 0 & 0 & -3
             \end{bmatrix}.
\end{align*}
Moreover,
$$
\left(\textnormal{tr}A_1\right)^2 + \left(\textnormal{tr}A_2\right)^2 + \left(\textnormal{tr}A_3\right)^2 = 1 \quad \textnormal{and} \quad \textnormal{tr}S = 7.
$$

Next, in order to use the biharmonicity condition, we replace the above results in expression of the first component of $\tau_2(\phi)$, given by Theorem \ref{t4}. We obtain
\begin{align*}
\left(\tau_2(\phi)\right)^1 =& -4\left(7-4\left(2(x^1)^2 + 2(x^2)^2 + 3y^2 \right)\right)\\
                             & + \left(140 - 168\left(2(x^1)^2 + 2(x^2)^2 + 3y^2 \right) + 32\left(2(x^1)^2 + 2(x^2)^2 + 3y^2 \right)^2\right)\cdot\\
                             & \cdot \left((x^1)^2 + (x^2)^2 - y^2 \right) + 32\left(2(x^1)^2 + 2(x^2)^2 - 3y^2 \right).
\end{align*}
Evaluating $\left(\tau_2(\phi)\right)^1$ at $(\overline x,y) = \left(1/2, 1/2, \sqrt{2}/2\right)$, we get
$$
\left(\tau_2(\phi)\right)^1\left(\frac{1}{2}, \frac{1}{2}, \frac{\sqrt{2}}{2}\right) = -4,
$$
thus we conclude that $\phi$ cannot be biharmonic.

For \emph{Case I.2}, we denote
$$
\left(F^2,F^3\right) = B + 2G,
$$
where $B$ is an $\r^2$-valued quadratic form in $\overline y = \left(y^1, y^2\right)$ alone and $G$ is an $\r^2$-valued bilinear form in $x$ and $\overline y$, such that
\begin{align}\label{w5}
|B|^2 + \left(\mu_1(y^1)^2 + \mu_2(y^2)^2\right)^2  = |y|^4 \nonumber\\
\langle B,G \rangle = 0 \\
2|G|^2 + x^2\left(\mu_1(y^1)^2 + \mu_2(y^2)^2\right) = x^2|y|^2, \nonumber
\end{align}
for any $(x,y)\in\r^3$.
Thus, $B$ and $G$ are given by
$$
B = \left(c^1\left(y^1\right)^2 + c^2\left(y^1\right)^2 + 2a^1y^1y^2, c^3\left(y^1\right)^2 + c^4\left(y^1\right)^2 + 2a^2y^1y^2\right)
$$
and
$$
G = \left(\gamma^1xy^1 + \gamma^2xy^2,\gamma^3xy^1 + \gamma^4xy^2\right),
$$
where $c^1$, $c^2$, $c^3$, $c^4$, $a^1$, $a^2$ and $\gamma^1$, $\gamma^2$, $\gamma^3$, $\gamma^4$ are real constants.
With the above notation, from Equation \eqref{w3}, comparing the coefficients, we obtain
\begin{align}\label{w6}
&\left(c^1\right)^2 + \left(c^3\right)^2 = 1 - \mu_1^2,        & \left(c^2\right)^2 + \left(c^4\right)^2 = 1 - \mu_2^2, \nonumber\\
&2\left(a^1\right)^2 + 2\left(a^2\right)^2 + c^1c^2 + c^3c^4 = 1 - \mu_1\mu_2,       & a^1c^2 + a^2c^4 = 0, \nonumber \\
&c^2\gamma^1 + c^4\gamma^3 + 2a^1\gamma^2 + 2a^2\gamma^4 = 0,  & a^1c^1 + a^2c^3 = 0, \\
&c^1\gamma^2 + c^3\gamma^4 + 2a^1\gamma^1 + 2a^2\gamma^3 = 0,  & c^1\gamma^1 + c^3\gamma^3 = 0, \nonumber\\
&\left(\gamma^1\right)^2 + \left(\gamma^3\right)^2 = \frac{1-\mu_1}{2},              & c^2\gamma^2 + c^4\gamma^4 = 0, \nonumber \\
&\left(\gamma^2\right)^2 + \left(\gamma^4\right)^2 = \frac{1-\mu_2}{2},              & \gamma^1\gamma^2 + \gamma^3\gamma^4 = 0. \nonumber
\end{align}

By direct computations we get that the only solution of \eqref{w6} is $\mu_1 = \mu_2 = -1$ and
\begin{align*}
&c^1 = c^2 = c^3 = c^4 = 0, \\
&a^1 = a^2 = 0, \\
&\gamma^1 = \cos \theta, \ \gamma^2 = \sin \theta, \\
&\gamma^3 = -\sin \theta, \ \gamma^4 = \cos \theta.
\end{align*}
Thus, the map $F$ takes the form
$$
F\left(x,y^1,y^2\right) =  \left(x^2 - \left(y^1\right)^2 - \left(y^2\right)^2, 2\cos\theta xy^1 + 2\sin\theta, -2\sin\theta xy^1 + 2\cos\theta xy^2\right).
$$
As $\phi$ is the restriction of Hopf fibration we conclude, as in the first case, that $\phi$ cannot be biharmonic.

\textit{Case II:} $m > 2$.

\noindent For $m > 2$, we use the result in \cite{W68} that says that any quadratic form from $\s^m$ to $\s^n$, with $m\geq 2n$, is constant. It follows that the only interesting situation is given by $m = 3$. It was proved in \cite{C98} that any non-constant quadratic form from $\s^3$ and $\s^2$, up to isometries of the domain and codomain, is the Hopf fribration $\s^3\rightarrow\s^2$. But it is well known that the Hopf fibration is harmonic.
\end{proof}

\begin{theorem}\label{t4.6}
Up to homothetic transformations of the domain and/or codomain, the only proper biharmonic quadratic form from $\s^m$ to $\s^3$, $m\geq2$, is the Hopf fibration $\psi:\s^3\rightarrow\s^2$ followed by the inclusion, as described in example \eqref{ec1}.
\end{theorem}
\begin{proof}

As for the previous theorem we split the proof in three main cases.

\textit{Case I:} $m = 2$.

\noindent Using the same approach as in the second method of \textit{Case II} below and Theorem 4.1, we can prove that there is no biharmonic map from $\s^2$ to $\s^3$.

\textit{Case II:} $m = 3$.

\noindent We can give a short proof of this case using the classification result of the non-constant quadratic forms from $\s^3$ to $\s^3$, obtained in \cite{Y86}.

Indeed, acording to \cite{Y86}, any non-constant quadratic form from $\s^3$ to $\s^3$ is, up to isometries of the domain and codomain, either the Hopf construction $\s^3\rightarrow\s^3$, or the composition between the Hopf fibration $\s^3\rightarrow\s^2$ and, after modifying the radius, the canonical inclusion $\s^2(a)\rightarrow\s^3$, $a\in(0,1]$. Using Theorem \ref{t4}, one can prove that the Hopf construction cannot be biharmonic. For the second map, we can use Theorem \ref{t4} or Theorem 2.1 in \cite{O03} to prove that it is proper biharmonic if and only if $a = 1/\sqrt{2}$.

As in the previous proof, for the sake of completeness, we present here a more direct proof of our theorem. Let $F:\r^4\rightarrow\r^4$ be a quadratic form such that its restriction $\phi:\s^3\rightarrow\s^3$ is not constant. Using \cite{W68}, we distinguish three subcases:

\emph{Case II.1}
\begin{align*}
F\left(x^1, x^2, x^3, y\right) =& \left(\left(x^1\right)^2 + \left(x^2\right)^2 + \left(x^3\right)^2 + \mu y^2,\right. \\
                     & \left. c^1 y^2 + 2\gamma_1^1 x^1y + 2\gamma_2^1 x^2y + 2\gamma_3^1 x^3y,\right.\\
                     & \left. c^2 y^2 + 2\gamma_1^2 x^1y + 2\gamma_2^2 x^2y + 2\gamma_3^2 x^3y,\right.\\
                     & \left. c^3 y^2 + 2\gamma_1^3 x^1y + 2\gamma_2^3 x^2y + 2\gamma_3^3 x^3y\right),
\end{align*}
for any $\left(x^1,x^2,x^3,y\right)\in\r^4$, with $-1\leq \mu <1$,

\emph{Case II.2}
\begin{align*}
F&\left(x^1,x^2,y^1,y^2\right) =\\
            =&  \left(\left(x^1\right)^2 + \left(x^2\right)^2 + \mu_1\left(y^1\right)^2  + \mu_2\left(y^2\right)^2,\right. \\
             & \left. c_1^1 \left(y^1\right)^2 + c_2^1 \left(y^2\right)^2 + 2a^1 y^1y^2 + 2\gamma_1^1 x^1y^1 + 2\gamma_2^1x^1y^2 + 2\gamma_3^1 x^2y^1 + 2\gamma_4^1x^2y^2,\right.\\
             & \left. c_1^2 \left(y^1\right)^2 + c_2^2 \left(y^2\right)^2 + 2a^2 y^1y^2 + 2\gamma_1^2 x^1y^1 + 2\gamma_2^2x^1y^2 + 2\gamma_3^2 x^2y^1 + 2\gamma_4^2x^2y^2,\right.\\
             & \left. c_1^3 \left(y^1\right)^2 + c_2^3 \left(y^2\right)^2 + 2a^3 y^1y^2 + 2\gamma_1^3 x^1y^1 + 2\gamma_2^3x^1y^2 + 2\gamma_3^3 x^2y^1 + 2\gamma_4^3x^2y^2\right),
\end{align*}
for any $\left(x^1, x^2, y^1, y^3\right)\in\r^4$, with  $-1\leq \mu_1 <1$ and  $-1\leq \mu_2 <1$,

\emph{Case II.3}
\begin{align*}
F\left(x,y^1,y^2,y^3\right) = & \left(x^2 + \mu_1\left(y^1\right)^2  + \mu_2\left(y^2\right)^2 + \mu_3\left(y^3\right)^2,\right. \\
             & \left. c_1^1 \left(y^1\right)^2 + c_2^1 \left(y^2\right)^2 + c_3^1 \left(y^3\right)^2 + 2a_1^1 y^1y^2 + 2a_2^1 y^1y^3 + 2a_3^1y^2y^3 \right.\\
             & \left. + 2\gamma_1^1 xy^1 + 2\gamma_2^1xy^2 + 2\gamma_3^1 xy^3,\right.\\
             & \left. c_1^2 \left(y^1\right)^2 + c_2^2 \left(y^2\right)^2 + c_3^2 \left(y^3\right)^2 + 2a_1^2 y^1y^2 + 2a_2^2 y^1y^3 + 2a_3^2y^2y^3 \right. \\
             & \left. + 2\gamma_1^2 xy^1 + 2\gamma_2^2xy^2 + 2\gamma_3^2 xy^3,\right.\\
             & \left. c_1^3 \left(y^1\right)^2 + c_2^3 \left(y^2\right)^2 + c_3^3 \left(y^3\right)^2 + 2a_1^3 y^1y^2 + 2a_2^3 y^1y^3 + 2a_3^3y^2y^3 \right. \\
             & \left. + 2\gamma_1^3 xy^1 + 2\gamma_2^3xy^2 + 2\gamma_3^3 xy^3\right),
\end{align*}
for any $\left(x^1, x^2, y^1, y^3\right)\in\r^4$, with  $-1\leq \mu_1 <1$, $-1\leq \mu_2 <1$  and  $-1\leq \mu_3 <1$.

For each subcase we will first determine all the corresponding maps $F$ and then we will check the biharmonicity condition from Theorem \ref{t4} or Proposition \ref{p2}.

For \emph{Case II.1}, we denote
$$
\left(F^2,F^3,F^4\right) = B + 2G,
$$
where $B$ is an $\r^3$-valued quadratic form in $y$ alone and $G$ is an $\r^3$-valued bilinear form in $\overline x = \left(x^1, x^2, x^3\right)$ and $y$, such that
\begin{equation}\label{w7}
|B|^2 + \mu^2y^4 = y^4, \quad \langle B,G \rangle = 0, \quad \textnormal{and} \quad 2|G|^2 + \mu |x|^2y^2 = |x|^2y^2,
\end{equation}
for any $(\overline x,y)\in\r^4$.
Thus, $B$ and $G$ are given by
$$
B = \left(c^1y^2, c^2y^2, c^3y^3\right)
$$
and
$$
G = \left(\gamma_1^1x^1y + \gamma_2^1x^2y + \gamma_3^1x^3y, \gamma_1^2x^1y + \gamma_2^2x^2y + \gamma_3^2x^3y, \gamma_1^3x^1y + \gamma_2^3x^2y + \gamma_3^3x^3y\right),
$$
where $c^i$ and $\gamma_j^i$, for $i,j\in\{1,2,3\}$, are real constants.
With the above notation, from Equation \eqref{w7}, comparing the coefficients, we obtain
\begin{align}\label{w8}
& \left(c^2\right)^2 + \left(c^2\right)^2 + \left(c^3\right)^2 = 1 - \mu^2, & c^1\gamma_1^1 + c^2\gamma_1^2 + c^3\gamma_1^3 = 0, \nonumber\\
&\left(\gamma_2^1\right)^2 + \left(\gamma_2^2\right)^2 + \left(\gamma_2^3\right)^2 = \frac{1-\mu}{2},  & c^1\gamma_3^1 + c^2\gamma_3^2 + c^3\gamma_3^3 = 0, \nonumber\\
& \left(\gamma_1^1\right)^2 + \left(\gamma_1^2\right)^2 + \left(\gamma_1^3\right)^2 = \frac{1-\mu}{2}, & c^1\gamma_2^1 + c^2\gamma_2^2 + c^3\gamma_2^3 = 0, \\
& \left(\gamma_3^1\right)^2 + \left(\gamma_3^2\right)^2 + \left(\gamma_3^3\right)^2 = \frac{1-\mu}{2}, & \gamma_1^1\gamma_2^1 + \gamma_1^2\gamma_2^2 + \gamma_1^3\gamma_2^3 = 0, \nonumber\\
& \gamma_1^1\gamma_3^1 + \gamma_1^2\gamma_3^2 + \gamma_1^3\gamma_3^3 = 0,   & \gamma_2^1\gamma_3^1 + \gamma_2^2\gamma_3^2 + \gamma_2^3\gamma_3^3 = 0 \nonumber.
\end{align}

If we denote the vectors
$$
\overline c = (c^1, c^2, c^3), \ \overline \gamma_1 = (\gamma_1^1, \gamma_1^2, \gamma_1^3), \ \overline \gamma_2 = (\gamma_2^1, \gamma_2^2, \gamma_2^3), \ \overline \gamma_3 = (\gamma_3^1, \gamma_3^2, \gamma_3^3),
$$
by looking at \eqref{w8} we observe that
$$
\overline c\perp \overline \gamma_1, \ \overline c\perp \overline \gamma_2, \ \overline c\perp \overline \gamma_3
$$
and
$$
\overline \gamma_1\perp \overline \gamma_2, \ \overline \gamma_1\perp \overline \gamma_3, \ \overline \gamma_2\perp \overline \gamma_3.
$$
Since $-1\leq\mu<1$, from the same equations it follows that $\overline \gamma_1 \neq \overline{0}$, $\overline \gamma_2 \neq \overline{0}$ and $\overline \gamma_3 \neq \overline{0}$. These imply that $\overline c = \overline{0}$, and further, because $|\overline c|^2 = 1 - \mu^2$, we obtain $\mu = -1$.

Acting with an isometry of $\r^3\equiv\{0\}\times\r^3$, the map $F$ keeps the same type of expression and we can assume that the vectors $\overline \gamma_i$ have the form
$$
\overline \gamma_1 = (1,0,0), \quad \overline \gamma_2 = (0,1,0), \quad \overline \gamma_3 = (0,0,1),
$$
and thus
$$
F\left(x^1, x^2, x^3, y\right) = \left(\left(x^1\right)^2 + \left(x^2\right)^2 + \left(x^3\right)^2 - y^2, 2x^1y , 2x^2y, 2x^3y\right).
$$
We can see that the map $\phi$ is nothing but the restriction to $\s^3$ of the quaternionic multiplication $q\rightarrow q^2$, i.e.
$\phi$ is the Hopf construction coresponding to the quaternions.

The map $F$ is given by the matrices

\begin{align*}
  A_1 = \begin{bmatrix}
             1 & 0 & 0 & 0\\
             0 & 1 & 0 & 0\\
             0 & 0 & 1 & 0\\
             0 & 0 & 0 & -1
         \end{bmatrix},
\quad
  A_2 = \begin{bmatrix}
             0 & 0 & 0 & 1\\
             0 & 0 & 0 & 0\\
             0 & 0 & 0 & 0\\
             1 & 0 & 0 & 0
         \end{bmatrix},
\end{align*}
\begin{align*}
  A_3 = \begin{bmatrix}
             0 & 0 & 0 & 0\\
             0 & 0 & 0 & 1\\
             0 & 0 & 0 & 0\\
             0 & 1 & 0 & 0
         \end{bmatrix},
\quad
  A_4 = \begin{bmatrix}
             0 & 0 & 0 & 0\\
             0 & 0 & 0 & 0\\
             0 & 0 & 0 & 1\\
             0 & 0 & 1 & 0
         \end{bmatrix},
\end{align*}
Next, since
\begin{align*}
  A_1^2 = \begin{bmatrix}
             1 & 0 & 0 & 0\\
             0 & 1 & 0 & 0\\
             0 & 0 & 1 & 0\\
             0 & 0 & 0 & 1
         \end{bmatrix},
\quad
  A_2^2 = \begin{bmatrix}
             1 & 0 & 0 & 0\\
             0 & 0 & 0 & 0\\
             0 & 0 & 0 & 0\\
             0 & 0 & 0 & 1
         \end{bmatrix},
\end{align*}
\begin{align*}
  A_3^2 = \begin{bmatrix}
             0 & 0 & 0 & 0\\
             0 & 1 & 0 & 0\\
             0 & 0 & 0 & 0\\
             0 & 0 & 0 & 1
         \end{bmatrix},
\quad
  A_4^2 = \begin{bmatrix}
             0 & 0 & 0 & 0\\
             0 & 0 & 0 & 0\\
             0 & 0 & 1 & 0\\
             0 & 0 & 0 & 1
         \end{bmatrix},
\end{align*}
using Equations \eqref{w4}. it follows that
\begin{align*}
  S = \begin{bmatrix}
                 2 & 0 & 0 & 0 \\
                 0 & 2 & 0 & 0 \\
                 0 & 0 & 2 & 0 \\
                 0 & 0 & 0 & 4
        \end{bmatrix}
\quad \textnormal{ and } \quad
A_1S = \begin{bmatrix}
                 2 & 0 & 0 & 0 \\
                 0 & 2 & 0 & 0 \\
                 0 & 0 & 2 & 0 \\
                 0 & 0 & 0 & -4
           \end{bmatrix}.
\end{align*}
Moreover,
$$
\left(\textnormal{tr}A_1\right)^2 + \left(\textnormal{tr}A_2\right)^2 + \left(\textnormal{tr}A_3\right)^2 + \left(\textnormal{tr}A_4\right)^2 = 4 \quad \textnormal{and} \quad \textnormal{tr}S = 10.
$$

Next, we replace the above results in expression of the first component of $\tau_2(\phi)$, given by Proposition \ref{p2}. We obtain
\begin{align*}
\left(\tau_2(\phi)\right)^1 =& -8\left(8 - 4\left(2(x^1)^2 + 2(x^2)^2 + 2(x^3)^2 + 4y^2 \right)\right)\\
                             & + \left(192 - 192\left(2(x^1)^2 + 2(x^2)^2 + 2(x^3)^2 + 4y^2 \right) \right. \\
                             & \left.+ 32\left(2(x^1)^2 + 2(x^2)^2 2(x^3)^2 + 4y^2 \right)^2\right) \left((x^1)^2 + (x^2)^2 + (x^3)^2 - y^2 \right)\\
                             & + 32\left(2(x^1)^2 + 2(x^2)^2 2(x^3)^2 - 4y^2 \right).
\end{align*}
Evaluating $\left(\tau_2(\phi)\right)^1$ at $(\overline x,y) = \left(1/2, 1/2, 1/2, 1/2\right)$, we get
$$
\left(\tau_2(\phi)\right)^1\left(\frac{1}{2}, \frac{1}{2}, \frac{1}{2}, \frac{1}{2}\right) = -12,
$$
thus we conclude that $\phi$ cannot be biharmonic.

For \emph{Case II.2}, we denote
$$
(F^2,F^3,F^4) = B + 2G,
$$
where $B$ is an $\r^3$-valued quadratic form in $\overline y = (y^1,y^2)$ alone and $G$ is an $\r^3$-valued bilinear form in $\overline x = (x^1, x^2)$ and $\overline y$, such that
\begin{align}\label{w9}
&|B|^2 + \left(\mu_1(y^1)^2 + \mu_2(y^2)^2\right)^2 = y^4,\nonumber\\
&\langle B,G \rangle = 0,\\
&2|G|^2 + |x|^2\left(\mu_1(y^1)^2 + \mu_2(y^2)^2\right) = |x|^2|y|^2,\nonumber
\end{align}
for any $(\overline x, \overline y)\in\r^4$.

Thus, $B$ and $G$ are given by
\begin{align*}
B =& \left( c_1^1 \left(y^1\right)^2 + c_2^1 \left(y^2\right)^2 + 2a^1 y^1y^2,\right.\\
   & \left. c_1^2 \left(y^1\right)^2 + c_2^2 \left(y^2\right)^2 + 2a^2 y^1y^2,\right.\\
   & \left. c_1^3 \left(y^1\right)^2 + c_2^3 \left(y^2\right)^2 + 2a^3 y^1y^2 \right)
\end{align*}
and
\begin{align*}
G =& \left( \gamma_1^1 x^1y^1 + \gamma_2^1x^1y^2 + \gamma_3^1 x^2y^1 + \gamma_4^1x^2y^2,\right.\\
   & \left. \gamma_1^2 x^1y^1 + \gamma_2^2x^1y^2 + \gamma_3^2 x^2y^1 + \gamma_4^2x^2y^2,\right.\\
   & \left. \gamma_1^3 x^1y^1 + \gamma_2^3x^1y^2 + \gamma_3^3 x^2y^1 + \gamma_4^3x^2y^2 \right),
\end{align*}
where $c_j^i$ and $a^i$, for $i,j\in\{1,2,3\}$, and $\gamma_j^i$, for $i\in\{1,2,3\}$ and $j\in\{1,2,3,4\}$, are real constants.
With the above notation, from Equation \eqref{w9}, comparing the coefficients, we obtain
\begin{align}\label{w10.1}
& (c_1^1)^2 + (c_1^2)^2 + (c_1^3)^2 = 1 - \mu_1^2,                        & (c_2^1)^2 + (c_2^2)^2 + (c_2^3)^2 = 1 - \mu_2^2, \nonumber\\
& a^1c_1^1 + a^2c_1^2 + a^3c_1^3 = 0,                                     &  a^1c_2^1 + a^2c_2^2 + a^3c_2^3 = 0, \nonumber\\
& c_1^1\gamma_1^1 + c_1^2\gamma_1^2 + c_1^3\gamma_1^3 = 0,                & c_1^1\gamma_3^1 + c_1^2\gamma_3^2 + c_1^3\gamma_3^3 = 0, \nonumber\\
& c_2^1\gamma_2^1 + c_2^2\gamma_2^2 + c_2^3\gamma_2^3 = 0,                & c_2^1\gamma_4^1 + c_2^2\gamma_4^2 + c_2^3\gamma_4^3 = 0, \\
& (\gamma_1^1)^2 + (\gamma_1^2)^2 + (\gamma_1^3)^2 = \frac{1-\mu_1}{2},   & (\gamma_2^1)^2 + (\gamma_2^2)^2 + (\gamma_2^3)^2 = \frac{1-\mu_2}{2}, \nonumber\\
& (\gamma_3^1)^2 + (\gamma_3^2)^2 + (\gamma_3^3)^2 = \frac{1-\mu_1}{2},   & (\gamma_4^1)^2 + (\gamma_4^2)^2 + (\gamma_4^3)^2 = \frac{1-\mu_2}{2}, \nonumber\\
&\gamma_1^1\gamma_2^1 + \gamma_1^2\gamma_2^2 + \gamma_1^3\gamma_2^3 = 0,  & \gamma_3^1\gamma_4^1 + \gamma_3^2\gamma_4^2 + \gamma_3^3\gamma_4^3 = 0, \nonumber \\
&\gamma_1^1\gamma_3^1 + \gamma_1^2\gamma_3^2 + \gamma_1^3\gamma_3^3 = 0,  & \gamma_2^1\gamma_4^1 + \gamma_2^2\gamma_4^2 + \gamma_2^3\gamma_4^3 = 0, \nonumber
\end{align}
and
\begin{align}\label{w10.2}
& 2(a^1)^2 + 2(a^2)^2 + 2(a^3)^2 + c_1^1c_2^1 + c_1^2c_2^2 + c_1^3c_2^3 = 1 - \mu_1\mu_2, \nonumber\\
& 2a^1\gamma_1^1 + 2a^2\gamma_1^2 + 2a^3\gamma_1^3 + c_1^1\gamma_2^1 + c_1^2\gamma_2^1 + c_1^3\gamma_2^3 = 0, \nonumber \\
& 2a^1\gamma_3^1 + 2a^2\gamma_3^2 + 2a^3\gamma_3^3 + c_1^1\gamma_4^1 + c_1^2\gamma_4^2 + c_1^3\gamma_4^3 = 0, \\
& 2a^1\gamma_2^1 + 2a^2\gamma_2^2 + 2a^3\gamma_2^3 + c_2^1\gamma_1^1 + c_2^2\gamma_1^2 + c_2^3\gamma_1^3 = 0, \nonumber\\
& 2a^1\gamma_4^1 + 2a^2\gamma_4^2 + 2a^3\gamma_4^3 + c_2^1\gamma_3^1 + c_2^2\gamma_3^2 + c_2^3\gamma_3^3 = 0, \nonumber \\
&\gamma_1^1\gamma_4^1 + \gamma_1^2\gamma_4^2 + \gamma_1^3\gamma_4^3 + \gamma_2^1\gamma_3^1 + \gamma_2^2\gamma_3^2 + \gamma_2^3\gamma_3^3 = 0. \nonumber
\end{align}

We denote the vectors $\overline c_i = \left(c_i^1, c_i^2, c_i^3\right)$, for $i = 1$, $2$, $\overline a = \left(a^1, a^2, a^3\right)$ and $\overline \gamma_i = \left(\gamma_i^1, \gamma_i^2, \gamma_i^3\right)$, for $i = 1$, $2$, $3$, $4$. Then, the systems of Equations \eqref{w10.1}, and \eqref{w10.2} can be rewritten as:
\begin{align}\label{w10.3}
& |\overline c_1|^2 = 1 - \mu_1^2, & |\overline c_2|^2 = 1 - \mu_2^2, \nonumber \\
& \overline a \perp \overline c_1, & \overline a \perp \overline c_2, \nonumber \\
& \overline c_1 \perp \overline\gamma_1, & \overline c_1 \perp \overline\gamma_3, \nonumber\\
& \overline c_2 \perp \overline\gamma_2, & \overline c_2 \perp \overline\gamma_4,\\
& |\overline \gamma_1|^2 = \frac{1-\mu_1}{2}, & |\overline \gamma_3|^2 = \frac{1-\mu_1}{2}, \nonumber \\
& |\overline \gamma_2|^2 = \frac{1-\mu_2}{2}, & |\overline \gamma_4|^2 = \frac{1-\mu_2}{2}, \nonumber \\
& \overline\gamma_1 \perp \overline\gamma_2, & \overline\gamma_1 \perp \overline\gamma_3, \nonumber \\
& \overline\gamma_4 \perp \overline\gamma_2, & \overline\gamma_4 \perp \overline\gamma_3, \nonumber
\end{align}
and
\begin{align}\label{w10.4}
2|\overline a|^2 + \langle\overline c_1, \overline c_2\rangle = 1 - \mu_1\mu_2, \nonumber \\
2\langle\overline a, \overline\gamma_1\rangle + \langle\overline c_1, \overline\gamma_2\rangle = 0, \nonumber \\
2\langle\overline a, \overline\gamma_3\rangle + \langle\overline c_1, \overline\gamma_4\rangle = 0, \\
2\langle\overline a, \overline\gamma_2\rangle + \langle\overline c_2, \overline\gamma_1\rangle = 0, \nonumber \\
2\langle\overline a, \overline\gamma_4\rangle + \langle\overline c_2, \overline\gamma_3\rangle = 0, \nonumber \\
\langle\overline\gamma_1, \overline\gamma_4\rangle + \langle\overline\gamma_2,\overline\gamma_3\rangle = 0. \nonumber
\end{align}
Next, if we denote $\theta_1 = \sphericalangle(\overline\gamma_2, \overline\gamma_3)$ and $\theta_2 = \sphericalangle(\overline\gamma_1, \overline\gamma_4)$, then from
$$
\langle\overline\gamma_1, \overline\gamma_4\rangle + \langle\overline\gamma_2,\overline\gamma_3\rangle = 0,
$$
we obtain
$$
\cos\theta_1 + \cos\theta_2 = 0.
$$
Taking in consideration the orthogonality relations between the vectors $\overline\gamma_i$, we conclude that we have only two cases: $\theta_1 = 0$ and $\theta_2 = \pi$, or $\theta_1 = \pi$ and $\theta_2 = 0$.
We will discuss only the case $\theta_1 = 0$ and $\theta_2 = \pi$, since the other cases leads to the same result.

Acting with an isometry of $\r^3\equiv\{0\}\times\r^3$, the map $F$ keeps the same type of expression and we can assume that the vectors $\overline \gamma_i$ have the form
\begin{align*}
& \overline \gamma_1 = (1,0,0)\sqrt{\frac{1-\mu_1}{2}}, & \overline \gamma_2 = (0,1,0)\sqrt{\frac{1-\mu_2}{2}},\\
& \overline \gamma_3 = (0,1,0)\sqrt{\frac{1-\mu_1}{2}}, & \quad \overline \gamma_4 = (-1,0,0)\sqrt{\frac{1-\mu_2}{2}}.
\end{align*}
Since, $\overline c_1 \perp \overline\gamma_1$ and $\overline c_1 \perp \overline\gamma_3$, it follows that $\pm\overline c_1 = (0,0,1)\sqrt{1-\mu_1^2}$. Similarly, $\pm\overline c_1 = (0,0,1)\sqrt{1-\mu_2^2}$. From the equation $a\perp\overline c_1$, and from equations
$$
2\langle\overline a, \overline\gamma_1\rangle + \langle\overline c_1, \overline\gamma_2\rangle = 0 \ \textnormal{ and } \
2\langle\overline a, \overline\gamma_3\rangle + \langle\overline c_1, \overline\gamma_4\rangle = 0,
$$
we obtain that $\overline a = \overline 0$. Moreover, using
$$
2|\overline a|^2 + \langle\overline c_1, \overline c_2\rangle = 1 - \mu_1\mu_2
$$
it follows that the third components of the vectors $\overline c_i$ must have the same sign. We will take them to be positive since if we take them to be negative it will lead us to the same result. Also, by direct computations we obtain $\mu_1 = \mu_2 = \mu$.

Thus, the map $F$ takes the form
\begin{align}\label{functie}
F\left(x^1,x^2,y^1,y^2\right) =&  \left(\left(x^1\right)^2 + \left(x^2\right)^2 + \mu\left(y^1\right)^2  + \mu\left(y^2\right)^2, \sqrt{2(1-\mu)}\left(x^1y^1 - x^2y^2\right)\right. \nonumber \\
             &\left. \sqrt{2(1-\mu)}\left(x^1y^2 + x^2y^1\right), \sqrt{1-\mu^2}\left(\left(y^1\right)^2 + \left(y^2\right)^2\right)\right)
\end{align}

It is not difficult to see that by acting on $F$ with the isometry
\begin{align*}
  T = \begin{bmatrix}
             \sqrt{\frac{1+\mu}{2}} & 0 & 0 & \sqrt{\frac{1-\mu}{2}} \\
             0 & 1 & 0 & 0 \\
             0 & 0 & 1 & 0 \\
             -\sqrt{\frac{1-\mu}{2}} & 0 & 0 & \sqrt{\frac{1+\mu}{2}}
           \end{bmatrix},
\end{align*}
the map $\phi$ is nothing but the Hopf fibration $\psi:\s^3\rightarrow\s^2(a)$ followed by the inclusion of $\s^2(a)$, thought of as a small hypersphere of $\s^3$, into the ambient space. Here $a = \sqrt{(1-\mu)/2}$, $a\in(0,1]$.

The map $F$, given in \eqref{functie} is given by the matrices
\begin{align*}
  A_1 = \begin{bmatrix}
             1 & 0 & 0 & 0 \\
             0 & 1 & 0 & 0 \\
             0 & 0 & \mu & 0 \\
             0 & 0 & 0 & \mu
           \end{bmatrix},
\quad
  A_2 = \frac{\sqrt{2(1-\mu)}}{2}\begin{bmatrix}
             0 & 0 & 1 & 0 \\
             0 & 0 & 0 & -1 \\
             1 & 0 & 0 & 0 \\
             0 & -1 & 0 & 0
           \end{bmatrix},
\end{align*}
\begin{align*}
  A_3 = \frac{\sqrt{2(1-\mu)}}{2}\begin{bmatrix}
             0 & 0 & 0 & 1 \\
             0 & 0 & 1 & 0 \\
             0 & 1 & 0 & 0 \\
             1 & 0 & 0 & 0
           \end{bmatrix},
\quad
  A_4 = \pm\sqrt{1-\mu^2}\begin{bmatrix}
             0 & 0 & 0 & 0 \\
             0 & 0 & 0 & 0 \\
             0 & 0 & 1 & 0 \\
             0 & 0 & 0 & 1
           \end{bmatrix}.
\end{align*}
We can easily observe that
$$
A_1^2 + A_2^2 + A_3^2 + A_4^2 = (2 - \mu) I_4,
$$
and thus the map $F$ satisfies the conditions for Proposition \ref{p2}. In this case, $m = 3$ and $\alpha = 2 - \mu$. Since $\Dz F \neq 0,$ from the conditions $m + 5 = 4\alpha$ and $\left|\Dz F\right|^2 = 2(m + 1)^2$, we conclude that the map $\phi$ is proper biharmonic if and only if $\mu = 0$, i.e. $a = 1/\sqrt{2}$. For $\mu = 0$, the map $\phi$ is the one given in example \eqref{ec1}.

For \emph{Case II.3}, we denote
$$
(F^2,F^3,F^4) = B + 2G,
$$
where $B$ is an $\r^3$-valued quadratic form in $\overline y = (y^1,y^2,y^3)$ alone and $G$ is an $\r^3$-valued bilinear form in $x$ and $\overline y$, such that
\begin{align}\label{w11}
&|B|^2 + \left(\mu_1(y^1)^2 + \mu_2(y^2)^2 + \mu_3(y^3)^2\right)^2 = y^4,\nonumber\\
&\langle B,G \rangle = 0,\\
&2|G|^2 + x^2\left(\mu_1(y^1)^2 + \mu_2(y^2)^2 + \mu_3(y^3)^2\right) = x^2|y|^2,\nonumber
\end{align}
for any $(x, \overline y)\in\r^4$.

Thus, $B$ and $G$ are given by
\begin{align*}
B =& \left( c_1^1 \left(y^1\right)^2 + c_2^1 \left(y^2\right)^2 + c_3^1 \left(y^3\right)^2 + 2a_1^1 y^1y^2 + 2a_2^1 y^1y^3 + 2a_3^1 y^2y^3,\right.\\
   & \left. c_1^2 \left(y^1\right)^2 + c_2^2 \left(y^2\right)^2 + c_3^2 \left(y^3\right)^2 + 2a_1^2 y^1y^2 + 2a_2^2 y^1y^3 + 2a_3^2 y^2y^3,\right.\\
   & \left. c_1^3 \left(y^1\right)^2 + c_2^3 \left(y^2\right)^2 + c_3^3 \left(y^3\right)^2 + 2a_1^3 y^1y^2 + 2a_2^3 y^1y^3 + 2a_3^3 y^2y^3 \right)
\end{align*}
and
$$
G = \left( \gamma_1^1 xy^1 + \gamma_2^1xy^2 + \gamma_3^1 xy^3, \gamma_1^2 xy^1 + \gamma_2^2xy^2 + \gamma_3^2 xy^3, \gamma_1^3 xy^1 + \gamma_2^3xy^2 + \gamma_3^3 xy^3 \right)
$$
where $c_j^i$, $a_i^j$ and $\gamma_i^j$, for $i,j\in\{1,2,3\}$, are real constants.
With the above notation, from Equation \eqref{w11}, comparing the coefficients, we obtain
\begin{align}\label{w11.1}
& (c_1^1)^2 + (c_1^2)^2 + (c_1^3)^2 = 1 - \mu_1^2,                        & a_2^1c_1^1 + a_2^2c_1^2 + a_2^3c_1^3 = 0, \nonumber\\
& (c_3^1)^2 + (c_3^2)^2 + (c_3^3)^2 = 1 - \mu_3^2,                        & a_1^1c_1^1 + a_1^2c_1^2 + a_1^3c_1^3 = 0, \nonumber \\
& (c_2^1)^2 + (c_2^2)^2 + (c_2^3)^2 = 1 - \mu_2^2,                        & a_1^1c_2^1 + a_1^2c_2^2 + a_1^3c_2^3 = 0, \\
& a_3^1c_2^1 + a_3^2c_2^2 + a_3^3c_2^3 = 0,                               & a_3^1c_3^1 + a_3^2c_3^2 + a_3^3c_3^3 = 0, \nonumber \\
& a_2^1c_3^1 + a_2^2c_3^2 + a_2^3c_3^3 = 0,                               & \nonumber
\end{align}
\begin{align}\label{w11.2}
& 2(a_1^1)^2 + 2(a_1^2)^2 + 2(a_1^3)^2 + c_1^1c_2^1 + c_1^2c_2^2 + c_1^3c_2^3 = 1 -\mu_1\mu_2, \nonumber \\
& 2(a_2^1)^2 + 2(a_2^2)^2 + 2(a_2^3)^2 + c_1^1c_3^1 + c_1^2c_3^2 + c_1^3c_3^3 = 1 -\mu_1\mu_3, \nonumber \\
& 2(a_3^1)^2 + 2(a_3^2)^2 + 2(a_3^3)^2 + c_2^1c_3^1 + c_2^2c_3^2 + c_2^3c_3^3 = 1 -\mu_2\mu_3, \\
& a_3^1c_1^1 + a_3^2c_1^2 + a_3^3c_1^3 + a_1^1a_2^1 + a_1^2a_2^2 + a_1^3a_2^3 = 0, \nonumber \\
& a_1^1c_3^1 + a_1^2c_3^2 + a_1^3c_3^3 + a_2^1a_3^1 + a_2^2a_3^2 + a_2^3a_3^3 = 0, \nonumber\\
& a_2^1c_2^1 + a_2^2c_2^2 + a_2^3c_2^3 + a_1^1a_3^1 + a_1^2a_3^2 + a_1^3a_3^3 = 0, \nonumber
\end{align}
\begin{align}\label{w11.3}
&(\gamma_1^1)^2 + (\gamma_1^2)^2 + (\gamma_1^3)^2 = \frac{1-\mu_1}{2},   & c_2^1\gamma_2^1 + c_2^2\gamma_2^2 + c_2^3\gamma_2^3 = 0, \nonumber\\
&(\gamma_2^1)^2 + (\gamma_2^2)^2 + (\gamma_2^3)^2 = \frac{1-\mu_2}{2},   & c_1^1\gamma_1^1 + c_1^2\gamma_1^2 + c_1^3\gamma_1^3 = 0, \nonumber\\
&(\gamma_3^1)^2 + (\gamma_3^2)^2 + (\gamma_3^3)^2 = \frac{1-\mu_3}{2},   & c_3^1\gamma_3^1 + c_3^2\gamma_3^2 + c_3^3\gamma_3^3 = 0, \\
&\gamma_1^1\gamma_2^1 + \gamma_1^2\gamma_2^2 + \gamma_1^3\gamma_2^3 = 0, & \gamma_2^1\gamma_3^1 + \gamma_2^2\gamma_3^2 + \gamma_2^3\gamma_3^3 = 0, \nonumber \\
&\gamma_1^1\gamma_3^1 + \gamma_1^2\gamma_3^2 + \gamma_1^3\gamma_3^3 = 0, &  \nonumber
\end{align}
and
\begin{align}\label{w11.4}
& 2a_1^1\gamma_1^1 + 2a_1^2\gamma_1^2 + 2a_1^3\gamma_1^3 + c_1^1\gamma_2^1 + c_1^2\gamma_2^1 + c_1^3\gamma_2^3 = 0 \nonumber \\
& 2a_2^1\gamma_1^1 + 2a_2^2\gamma_1^2 + 2a_2^3\gamma_1^3 + c_1^1\gamma_3^1 + c_1^2\gamma_3^1 + c_1^3\gamma_3^3 = 0 \nonumber \\
& 2a_1^1\gamma_2^1 + 2a_1^2\gamma_2^2 + 2a_1^3\gamma_2^3 + c_2^1\gamma_1^1 + c_2^2\gamma_1^1 + c_2^3\gamma_1^3 = 0 \nonumber \\
& 2a_3^1\gamma_2^1 + 2a_3^2\gamma_2^2 + 2a_3^3\gamma_2^3 + c_2^1\gamma_3^1 + c_2^2\gamma_3^1 + c_2^3\gamma_3^3 = 0 \\
& 2a_2^1\gamma_3^1 + 2a_2^2\gamma_3^2 + 2a_2^3\gamma_3^3 + c_3^1\gamma_1^1 + c_3^2\gamma_1^1 + c_3^3\gamma_1^3 = 0 \nonumber \\
& 2a_3^1\gamma_3^1 + 2a_3^2\gamma_3^2 + 2a_3^3\gamma_3^3 + c_3^1\gamma_2^1 + c_3^2\gamma_2^1 + c_3^3\gamma_2^3 = 0 \nonumber \\
& a_3^1\gamma_1^1 + a_3^2\gamma_1^2 + a_3^3\gamma_1^3 + a_2^1\gamma_2^1 + a_2^2\gamma_2^2 + a_2^3\gamma_2^3 + a_1^1\gamma_3^1 + a_1^2\gamma_3^2 + a_1^3\gamma_3^3 = 0 \nonumber
\end{align}

We denote the vectors $\overline c_i = \left(c_i^1, c_i^2, c_i^3\right)$, $\overline a_i = \left(a_i^1, a_i^2, a_i^3\right)$ and $\overline \gamma_i = \left(\gamma_i^1, \gamma_i^2, \gamma_i^3\right)$, for $i = 1$, $2$, $3$. Then, the systems of Equations \eqref{w11.1}, \eqref{w11.2}, \eqref{w11.3} and \eqref{w11.4} can be rewritten as:
\begin{align}\label{w12.1}
&|\overline c_1|^2 = 1 - \mu_1^2, & |\overline c_2|^2 = 1 - \mu_2^2, \nonumber\\
&|\overline c_3|^2 = 1 - \mu_3^2, & \nonumber\\
&\overline a_1 \perp \overline c_1, & \overline a_1 \perp \overline c_2, \\
&\overline a_2 \perp \overline c_1, & \overline a_2 \perp \overline c_3, \nonumber\\
&\overline a_3 \perp \overline c_2, & \overline a_3 \perp \overline c_3, \nonumber
\end{align}

\begin{align}\label{w12.2}
&2|\overline a_1|^2 + \langle\overline c_1, \overline c_2\rangle  = 1 - \mu_1\mu_2, & \langle\overline a_3, \overline c_1\rangle + \langle\overline a_1, \overline a_2\rangle = 0, \nonumber \\
&2|\overline a_2|^2 + \langle\overline c_1, \overline c_3\rangle  = 1 - \mu_1\mu_3, & \langle\overline a_1, \overline c_3\rangle + \langle\overline a_2, \overline a_3\rangle = 0, \\
&2|\overline a_3|^2 + \langle\overline c_2, \overline c_3\rangle  = 1 - \mu_2\mu_3, & \langle\overline a_2, \overline c_2\rangle + \langle\overline a_1, \overline a_3\rangle = 0, \nonumber
\end{align}

\begin{align}\label{w12.3}
&2\langle\overline a_1, \overline \gamma_1\rangle + \langle\overline c_1, \overline \gamma_2\rangle = 0, &2\langle\overline a_2, \overline \gamma_1\rangle + \langle\overline c_1, \overline \gamma_3\rangle = 0, \nonumber\\
&2\langle\overline a_1, \overline \gamma_2\rangle + \langle\overline c_2, \overline \gamma_1\rangle = 0, &2\langle\overline a_3, \overline \gamma_2\rangle + \langle\overline c_2, \overline \gamma_3\rangle = 0,\\
&2\langle\overline a_2, \overline \gamma_3\rangle + \langle\overline c_3, \overline \gamma_1\rangle = 0, &2\langle\overline a_3, \overline \gamma_3\rangle + \langle\overline c_3, \overline \gamma_2\rangle = 0,\nonumber\\
&\langle\overline a_3, \overline \gamma_1\rangle + \langle\overline a_2, \overline \gamma_2\rangle + \langle\overline a_1, \overline \gamma_3\rangle = 0, \nonumber
\end{align}

\begin{align}\label{w12.4}
&|\overline \gamma_1|^2 = \frac{1-\mu_1}{2}, & |\overline \gamma_2|^2 = \frac{1-\mu_2}{2}, \nonumber\\
&|\overline \gamma_3|^2 = \frac{1-\mu_3}{2}, &  \nonumber\\
&\overline \gamma_1 \perp \overline \gamma_2, & \overline \gamma_2 \perp \overline \gamma_3, \\
&\overline c_1 \perp \overline \gamma_1, & \overline \gamma_1 \perp \overline \gamma_3, \nonumber\\
&\overline c_2 \perp \overline \gamma_2, &\overline c_3 \perp \overline \gamma_3. \nonumber
\end{align}

Acting with an isometry of $\r^3\equiv\{0\}\times\r^3$, the map $F$ keeps the same type of expression and we can assume that the vectors $\overline \gamma_i$ have the form
$$
\overline \gamma_1 = \sqrt{\frac{1-\mu_1}{2}}(1,0,0), \quad \overline \gamma_2 = \sqrt{\frac{1-\mu_2}{2}}(0,2,0), \quad \overline \gamma_3 = \sqrt{\frac{1-\mu_3}{2}}(0,0,1),
$$
and thus
$$
\overline c_1 = \left(0, c_1^2, c_1^3\right), \quad \overline c_2 = \left(c_2^1, 0, c_2^3\right), \quad \overline c_3 = \left(c_3^1, c_3^2, 0\right).
$$
At this point we distinguish the following subcases:
\begin{itemize}
  \item[i)] $\mu_1 \neq -1$, $\mu_2 \neq -1$ and $\mu_3 \neq -1$,
  \item[ii)] $\mu_1 = -1$, $\mu_2 \neq -1$ and $\mu_3 \neq -1$ (and other permutations of this type),
  \item[iii)] $\mu_1 = \mu_2 = -1$ and $\mu_3 \neq -1$ (and other permutations of this type),
  \item[iv)] $\mu_1 = \mu_2 = \mu_3 = -1$.
\end{itemize}

For \emph{Case II.3.i}, it is easy to prove that $\overline a_i \neq 0$, for $i = 1$, $2$, $3$.
Next, since $\mu_i \neq 0$, it follows from system \eqref{w12.1} that we can rewrite $\overline c_k$ as
\begin{align*}
  \overline c_1 =& \sqrt{1-\mu_1^2}(0,\sin\theta_1,\cos\theta_1), \\
  \overline c_2 =& \sqrt{1-\mu_2^2}(\sin\theta_2,0,\cos\theta_2), \\
  \overline c_3 =& \sqrt{1-\mu_3^2}(\sin\theta_3,\cos\theta_3,0).
\end{align*}
Using Equations \eqref{w12.3}, by direct computations, we get:
\begin{align*}
& a_1^1 = -\frac{1}{2}\sin\theta_1\sqrt{1+\mu_1}\sqrt{1-\mu_2}, & a_1^2 = -\frac{1}{2}\sin\theta_2\sqrt{1+\mu_2}\sqrt{1-\mu_1}, \\
& a_2^1 = -\frac{1}{2}\cos\theta_1\sqrt{1+\mu_1}\sqrt{1-\mu_3}, & a_2^3 = -\frac{1}{2}\sin\theta_3\sqrt{1+\mu_3}\sqrt{1-\mu_1}, \\
& a_3^2 = -\frac{1}{2}\cos\theta_2\sqrt{1+\mu_2}\sqrt{1-\mu_3}, & a_3^3 = -\frac{1}{2}\cos\theta_3\sqrt{1+\mu_3}\sqrt{1-\mu_2}.
\end{align*}
If we suppose that $\sin\theta_i \neq 0$ and $\cos\theta_i \neq  0$, for any $i = 1$, $2$, $3$, then from system \eqref{w12.1}, we get
\begin{align}\label{w12.5}
  a_1^3 = &  \frac{\sin\theta_1\sin\theta_2}{2\cos\theta_2}\sqrt{1+\mu_1}\sqrt{1-\mu_2} = \frac{\sin\theta_1\sin\theta_2}{2\cos\theta_1}\sqrt{1+\mu_2}\sqrt{1-\mu_1}, \nonumber\\
  a_2^2 = & \frac{\sin\theta_3\cos\theta_1}{2\sin\theta_1}\sqrt{1+\mu_3}\sqrt{1-\mu_1} = \frac{\sin\theta_3\cos\theta_1}{2\cos\theta_3}\sqrt{1+\mu_1}\sqrt{1-\mu_3},\\
  a_3^1 = & \frac{\cos\theta_2\cos\theta_3}{2\sin\theta_2}\sqrt{1+\mu_3}\sqrt{1-\mu_2} = \frac{\cos\theta_2\cos\theta_3}{2\sin\theta_3}\sqrt{1+\mu_2}\sqrt{1-\mu_3}.\nonumber
\end{align}
From the above equalities, it follows that
\begin{align}\label{w13}
  \frac{\sqrt{1+\mu_1}\sqrt{1-\mu_2}}{\cos\theta_2} = & \frac{\sqrt{1+\mu_2}\sqrt{1-\mu_1}}{\cos\theta_1}, \nonumber\\
  \frac{\sqrt{1+\mu_3}\sqrt{1-\mu_1}}{\sin\theta_1} = & \frac{\sqrt{1+\mu_1}\sqrt{1-\mu_3}}{\cos\theta_3}, \\
  \frac{\sqrt{1+\mu_1}\sqrt{1-\mu_2}}{\cos\theta_2} = & \frac{\sqrt{1+\mu_2}\sqrt{1-\mu_1}}{\cos\theta_1}. \nonumber
\end{align}
Further, using Equations \eqref{w12.2} and \eqref{w13}, we get
\begin{align*}
  |\overline a_1|^2 = & \frac{1}{4}(1+\mu_2)(1-\mu_1)\left(\frac{\sin^2\theta_1}{\cos^2\theta_1} + \sin^2\theta_2\right), \\
  |\overline a_2|^2 = & \frac{1}{4}(1+\mu_1)(1-\mu_2)\left(\frac{\sin^2\theta_3}{\cos^2\theta_3} + \cos^2\theta_1\right),\\
  |\overline a_3|^2 = & \frac{1}{4}(1+\mu_2)(1-\mu_3)\left(\frac{\cos^2\theta_3}{\cos^2\theta_3} + \cos^2\theta_2\right).
\end{align*}

Using Equation \eqref{w12.5}, from
$$
2|\overline a_1|^2 + \langle\overline c_1, \overline c_2\rangle  = 1 - \mu_1\mu_2,
$$
see \eqref{w12.2}, we get
$$
\frac{1}{2}(1 + \mu_2)(1 - \mu_1)\left(\frac{1}{\cos^2\theta_1} + \cos^2\theta_2\right) = 1 - \mu_1\mu_2.
$$
Using again \eqref{w12.5}, we obtain the equation
$$
(1 + \mu_2)(1 - \mu_1)\cos^4\theta_2 - 2(1 - \mu_1\mu_2)\cos^2\theta_2 + (1 + \mu_1)(1 - \mu_2) = 0,
$$
which, for $\mu_2 > \mu_1$, has the solution
$$
\cos^2\theta_2 = \frac{(1 + \mu_1)(1 - \mu_2)}{(1 - \mu_1)(1 + \mu_2)}.
$$
In a similar way we obtain from \eqref{w12.3}
\begin{align*}
\frac{1}{2}(1 + \mu_2)(1 - \mu_3)\left(\frac{1}{\sin^2\theta_3} + \sin^2\theta_2\right) = 1 - \mu_2\mu_3,\\
\frac{1}{2}(1 + \mu_1)(1 - \mu_3)\left(\frac{1}{\cos^2\theta_3} + \sin^2\theta_1\right) = 1 - \mu_1\mu_3.
\end{align*}
and using Equations \eqref{w12.5}, for $\mu_2 > \mu_1 > \mu_3$ it follows that
$$
\sin^2\theta_2 = \frac{(1 - \mu_2)(1 + \mu_3)}{(1 + \mu_2)(1 - \mu_3)} \ \textnormal{ and } \ \sin^2\theta_1 = \frac{(1 - \mu_1)(1 + \mu_3)}{(1 + \mu_1)(1 - \mu_3)}.
$$
Using again Equations \eqref{w12.5}, we get
$$
\sin^2\theta_3 = \cos^2\theta_1 = \cos^2\theta_3 = 1
$$
which is in contradiction to our supposition that $\sin\theta_i \neq 0$ and $\cos\theta_i \neq  0$, for any $i = 1$, $2$, $3$.

If $\sin\theta_1 = 0,$ then
$$
\overline c_1 = \sqrt{1-\mu_1^2}(0,0,1).
$$
As $\overline c_1 \perp \overline a_1$ and $\overline c_1 \perp \overline a_2$ it follows that $a_1^3 = a_2^3 = 0$.
Since
$$
2\langle\overline a_2, \overline\gamma_3\rangle + \langle\overline c_3, \overline\gamma_1\rangle = 0,
$$
and
$$
2\langle\overline a_1, \overline\gamma_1\rangle + \langle\overline c_1, \overline\gamma_2\rangle = 0,
$$
we get that $\sin\theta_3 = 0$ and $a_1^1 = 0$.
As $\overline c_3 \perp \overline a_2$ and $\overline c_3 \perp \overline a_3$ it follows that $a_2^2 = a_3^2 = 0$.
Now, the vectors $\overline a_i$ have the following form
$$
\overline a_1 = \left(0,a_1^2, 0\right), \ \overline a_2 = \left(a_2^1,0,0\right) \ \textnormal{ and } \ \overline a_3 = \left(a_3^1,0,a_3^3\right).
$$
From the equation
$$
\langle\overline a_2, \overline c_2\rangle + \langle \overline a_1, \overline a_3\rangle = 0,
$$
see \eqref{w12.2}, since $\overline a_2 \neq \overline 0$, it follows that $\sin\theta_2 = 0$.
Using
$$
\langle\overline a_3, \overline \gamma_1\rangle + \langle\overline a_2, \overline \gamma_2\rangle + \langle\overline a_1, \overline \gamma_3\rangle = 0,
$$
see \eqref{w12.3}, we obtain $a_3^1 = 0$. Finally, since $\overline a_3 \perp \overline c_2$ it follows that $a_3^3\cos\theta_2 = 0$, but this is a contradiction considering that none of the factors can be zero. Thus our supposition that $\sin\theta_1 = 0$ is false.

The proof for the cases $\sin\theta_2 = 0$, or $\sin\theta_3 = 0$, or $\cos\theta_1 = 0$, or $\cos\theta_2 = 0$, or $\cos\theta_3 = 0$, follows in the same manner.

In conclusion, the Case i) leads us to contradiction.

For \emph{Case II.3.ii}, we immediately see that $\overline c_1 = \overline 0$, and thus, from Equations \eqref{w12.3} we get
$$
\overline a_1 \perp \overline a_2, \quad \overline a_1 \perp \overline\gamma_1, \quad \overline a_2 \perp \overline\gamma_1,
$$
and therefore $a_1^1 = a_2^1 = 0$.
Next, from Equations \eqref{w12.2} we get that $\overline a_1 \neq \overline 0$ and $\overline a_2 \neq \overline 0$.
From Equations \eqref{w12.3}, we immediately obtain that
\begin{align*}
& a_1^2 = -\sqrt{\frac{1+\mu_2}{2}}\sin\theta_2, & a_3^2 = -\frac{1}{2}\cos\theta_2\sqrt{1-\mu_3}\sqrt{1+\mu_2},\\
& a_2^3 = -\sqrt{\frac{1+\mu_3}{2}}\sin\theta_3, & a_3^3 = -\frac{1}{2}\cos\theta_3\sqrt{1-\mu_2}\sqrt{1+\mu_3},
\end{align*}
and, from Equations \eqref{w12.2}, we get
$$
\left(a_2^2\right)^2 = \frac{1+\mu_3}{2}\cos^2\theta_3 \ \textnormal{ and } \ \left(a_1^3\right)^2 = \frac{1+\mu_2}{2}\cos^2\theta_2.
$$
Since $\overline a_1 \perp \overline c_2$ and $\overline a_2 \perp \overline c_3$, it follows that $\cos\theta_2 = \cos\theta_3 = 0$. Using the last equation of system \eqref{w12.3}, we immediately get $a_3^1 = 0$.
From the equation
$$
2|\overline a_3|^2 + \langle\overline c_2, \overline c_3\rangle  = 1 - \mu_2\mu_3,
$$
we obtain that the third components of the vectors $\overline c_2$ and $\overline c_3$ must have the same sign, and, by direct computations we get $\mu_2 = \mu_3 = \mu$. Thus, the map $F$ takes the following form
\begin{align*}
F&\left(x,y^1,y^2,y^3\right) = \\
=& \left( x^2 - \left(y^1\right)^2 + \mu \left(y^2\right)^2 + \mu\left(y^3\right)^2, \sqrt{1-\mu^2}\left(y^2\right)^2 + \sqrt{1-\mu^2}\left(y^3\right)^2 + 2xy^1, \right. \\
& \left.-\sqrt{2(1+\mu)}y^1y^2 + \sqrt{2(1-\mu)}xy^2, -\sqrt{2(1+\mu)}y^1y^3 + \sqrt{2(1-\mu)}xy^3\right).
\end{align*}

As $\left|\dz F\right|^2$ is invariant under orthogonal linear transformations of $\r^4$, and looking at the squared norm of the differential of the Hopf construction $\s^3\rightarrow\s^3$, we can perform the following orthogonal transformation on the domain of $F$
\begin{align*}
  T = \begin{bmatrix}
             \sqrt{\frac{1-\mu}{2}} & -\sqrt{\frac{1+\mu}{2}} & 0 & 0 \\
             \sqrt{\frac{1+\mu}{2}} & \sqrt{\frac{1-\mu}{2}} & 0 & 0 \\
             0 & 0 & 1 & 0 \\
             0 & 0 & 0 & 1
      \end{bmatrix}.
\end{align*}
In this case,
\begin{align*}
\left(F\circ T^{-1}\right)&\left(\tilde x,\tilde y^1,\tilde y^2,\tilde y^3\right) = \\
= & \left( -\mu\tilde x^2 + \mu\left(\tilde y^1\right)^2 + \mu\left(\tilde y^2\right)^2 + \mu\left(\tilde y^3\right)^2 + 2\sqrt{1-\mu^2}\tilde x\tilde y^1,\right. \\
& \left. -\sqrt{1-\mu^2}\tilde x^2 + \sqrt{1-\mu^2}\left(\tilde y^1\right)^2 + \sqrt{1-\mu^2}\left(\tilde y^2\right)^2 + \sqrt{1-\mu^2}\left(\tilde y^3\right)^2 - 2\mu\tilde x\tilde y^1,\right. \\
& \left. 2\tilde x\tilde y^2, 2\tilde x\tilde y^3\right)
\end{align*}
Next we can apply the following linear orthogonal transformation on the components of $F\circ T^{-1}$
\begin{align*}
  T_1 = \begin{bmatrix}
             \mu & \sqrt{1-\mu^2} & 0 & 0 \\
             -\sqrt{1-\mu^2} & \mu & 0 & 0 \\
             0 & 0 & 1 & 0 \\
             0 & 0 & 0 & 1
      \end{bmatrix},
\end{align*}
and we obtain the map
\begin{align*}
\left(T_1\circ F\circ T^{-1}\right)\left(\tilde x,\tilde y^1,\tilde y^2,\tilde y^3\right) = & \left( -\tilde x^2 + \left(\tilde y^1\right)^2 + \left(\tilde y^2\right)^2 + \left(\tilde y^3\right)^2, - 2\tilde x\tilde y^1, 2\tilde x\tilde y^2, 2\tilde x\tilde y^3\right)
\end{align*}
which is the Hopf construction $\s^3\rightarrow\s^3$.
As in the \emph{Case II.1}, we conclude that $F$ cannot be biharmonic.

For \emph{Case II.3.iii}, from Equations \eqref{w12.1} we get
$$
\overline c_1 = \overline c_2 = \overline 0 \ \textnormal{ and } \ \overline c_3 = \sqrt{1-\mu_3^2}(\sin\theta,\cos\theta,0).
$$
From Equations \eqref{w12.2} and \eqref{w12.3} we immediately obtain
$$
\overline a_1 \perp \overline a_2, \ \overline a_1 \perp \overline a_3,
$$
and
$$
\overline a_1 \perp \overline\gamma_1, \ \overline a_1 \perp \overline\gamma_2, \ \overline a_2 \perp \overline\gamma_1, \ \overline a_3 \perp \overline\gamma_2.
$$
Now, the vectors $\overline a_i$ have the following form
$$
\overline a_1 = \left(0,0,a_1^3\right), \ \overline a_2 = \left(0,a_2^2,0\right) \ \textnormal{ and } \ \overline a_3 = \left(a_3^1,0,0\right).
$$
From equations
$$
2\langle\overline a_2, \overline \gamma_3\rangle + \langle\overline c_3, \overline \gamma_1\rangle = 0 \ \textnormal{ and } \ 2\langle\overline a_3, \overline \gamma_3\rangle + \langle\overline c_3, \overline \gamma_2\rangle = 0,
$$
see \eqref{w12.3}, we get that $\sin\theta = \cos\theta = 0$ which is a contradiction, since $\overline c_3 \neq \overline 0.$

For \emph{Case II.3.iv}, from Equations \eqref{w12.1}, \eqref{w12.2}, \eqref{w12.3} and \eqref{w12.4} we have
$$
\overline c_1 = \overline c_2 = \overline c_3 = \overline 0, \ |\overline a_1| = |\overline a_2| = |\overline a_3| = 1,
$$
$$
\overline \gamma_1 = (1,0,0), \ \overline \gamma_2 = (0,1,0), \ \overline \gamma_3 = (0,0,1),
$$
$$
\overline a_1 \perp \overline\gamma_1, \ \overline a_1 \perp \overline\gamma_2, \ \overline a_2 \perp \overline\gamma_1, \ \overline a_2 \perp \overline\gamma_3, \ \overline a_3 \perp \overline\gamma_2, \ \overline a_3 \perp \overline\gamma_3.
$$
These equations give a simple form for the vectors $\overline a_i$, but, further, the vectors $\overline a_i$ do not satisfy the equation
$$
\langle\overline a_3, \overline \gamma_1\rangle + \langle\overline a_2, \overline \gamma_2\rangle + \langle\overline a_1, \overline \gamma_3\rangle = 0,
$$
see \eqref{w12.3}. Thus, this subcase also leads us to a contradiction.

\emph{Case III} $m\geq 4$.

\noindent For $m > 5$, we use the result in \cite{W68} that says that any quadratic form from $\s^m$ to $\s^n$ with $m \geq 2n$ is constant. It follows that the only interesting situations are given by $m = 4$, $5$. But these cases are excluded by two results of \cite{C98}. The first one says that if $\phi:\s^{2n-1}\rightarrow\s^n$ is a non-constant map, then $n\in\{1,2,4,8\}$, and the second one says that if $\phi:\s^{2n-2}\rightarrow\s^n$ $(n \geq 2)$ is a non-constant map, then $n\in\{2,4,8\}$.
\end{proof}

\vspace{0.5cm}

All results obtained up to now suggest the following:

\vspace{0.3cm}
\noindent\textbf{\emph{Open Problem.}} If $\phi:\s^m\rightarrow\s^n$ is a proper biharmonic quadratic form then, up to an isometry of $\s^n$, the first $n$ components of $\phi$ are harmonic polynomials on $\r^{m+1}$ and form a map $\psi:\s^m\rightarrow\s^{n-1}(1/\sqrt{2})$.

\end{document}